\documentclass[11pt,a4paper]{amsart}
\usepackage{amssymb,amsmath,epsfig,graphics,mathrsfs}

\usepackage{fancyhdr}
\usepackage{hyperref}
\hypersetup{
 colorlinks   = true,
 urlcolor     = blue,
 linkcolor    = blue,
 citecolor   = red ,
 bookmarksopen=true
}

\usepackage{amsmath}
\usepackage{amsfonts}
\usepackage{amssymb}
\usepackage{amsthm}
\usepackage{epsfig,graphics,mathrsfs}
\usepackage{graphicx}
\usepackage{dsfont}

\usepackage[usenames, dvipsnames]{color} 

\usepackage{hyperref}

\newcommand*\MSC[1][1991]{\par\leavevmode\hbox{%
\textit{#1 Mathematical subject classification:\ }}}
\newcommand\blfootnote[1]{%
  \begingroup
  \renewcommand\thefootnote{}\footnote{#1}%
  \addtocounter{footnote}{-1}%
  \endgroup
}

\def \phi {\varphi}

\def \RN {\mathbb{R}^N}
\def \R {\mathbb{R}}

\def \G{\Gamma}

\def \vf{\varphi}

\newcommand{\Bps}{\mathfrak B_{s,p}(\bG)}

\newcommand{\Rn}{\mathbb R^n}
\newcommand{\Rm}{\mathbb R^m}
\newcommand{\Om}{\Omega}

\newcommand{\p}{\partial}

\newcommand{\bG}{\mathbb{G}}
\newcommand{\bg}{\mathfrak g}

\newcommand{\la}{\lambda}

\numberwithin{equation}{section}

\newcommand{\beq}{\begin{equation}}
\newcommand{\bea}[1]{\begin{array}{#1} }
\newcommand{\eeq}{ \end{equation}}
\newcommand{\ea}{ \end{array}}

\newcommand{\ve}{\varepsilon}

\newcommand{\In}{1_E}
\newcommand{\Lp}{L^p}

\newcommand{\nh}{\nabla_H}

\newtheorem{theorem}{Theorem}[section]
\newtheorem{lemma}[theorem]{Lemma}
\newtheorem{proposition}[theorem]{Proposition}

\newtheorem{definition}[theorem]{Definition}

\numberwithin{equation}{section}
\begin{document}

\title[A universal heat semigroup characterisation, etc.]{A universal heat semigroup characterisation\\of Sobolev and BV spaces in Carnot groups}

\blfootnote{\MSC[2020]{35K08, 46E35, 53C17}}
\keywords{Sub-Riemannian heat kernels, Integral decoupling, Folland-Stein and BV spaces}

\date{}

\begin{abstract}
In sub-Riemannian geometry there exist, in general, no known explicit representations of the heat kernels, and these functions fail to have any symmetry whatsoever. In particular, they are not a function of the control distance, nor they are for instance spherically symmetric in any of the layers of the Lie algebra. Despite these unfavourable aspects, in this paper we establish a new heat semigroup characterisation of the Sobolev and $BV$ spaces in a Carnot group by means of an integral decoupling property of the heat kernel.       
\end{abstract}

\author{Nicola Garofalo}

\address{Dipartimento d'Ingegneria Civile e Ambientale (DICEA)\\ Universit\`a di Padova\\ Via Marzolo, 9 - 35131 Padova,  Italy}
\vskip 0.2in
\email{nicola.garofalo@unipd.it}

\author{Giulio Tralli}
\address{Dipartimento d'Ingegneria Civile e Ambientale (DICEA)\\ Universit\`a di Padova\\ Via Marzolo, 9 - 35131 Padova,  Italy}
\vskip 0.2in
\email{giulio.tralli@unipd.it}

\maketitle

\tableofcontents

\section{Introduction}\label{S:intro}

For $1\le p < \infty$ and $0<s<1$ consider in $\Rn$ the Banach space $W^{s,p}$ of functions $f\in \Lp$ with finite Aronszajn-Gagliardo-Slobedetzky seminorm, 
\begin{equation}\label{ags}
[f]^p_{s,p} = \int_{\Rn} \int_{\Rn} \frac{|f(x) - f(y)|^p}{|x-y|^{n+ps}} dx dy,
\end{equation}
see e.g. \cite{Ad, RS}. In their celebrated works \cite{BBM1, BBM2, B}, Bourgain, Brezis and Mironescu discovered a new characterisation of the spaces $W^{1,p}$ and $BV$ based on the study of the limiting behaviour of the spaces $W^{s,p}$ as $s\nearrow 1$. To state their result, consider a one-parameter family of functions $\{\rho_\ve\}_{\ve>0}\in L^1_{loc}(0,\infty)$, $\rho_\ve\geq 0$, satisfying the following assumptions
\begin{equation}\label{condbbm}
\int_0^\infty \rho_\ve(r)r^{n-1}dr=1,\quad\underset{\ve \to 0^+}{\lim}\int_\delta^\infty \rho_\ve(r)r^{n-1}dr = 0\ \ \mbox{for every $\delta>0$},
\end{equation}
see \cite[(9)-(11)]{B}. Also, for $1\le p<\infty$ let
\[
K_{p,n}=\int_{\mathbb S^{n-1}} |\langle \omega,e_n\rangle|^p d\sigma(\omega).
\]

\vskip 0.3cm
 
\noindent \textbf{Theorem A.} [Bourgain, Brezis \& Mironescu]\label{T:bbm}\  
\emph{
Assume $1\le p <\infty$. Let $f\in L^p(\Rn)$ and suppose that 
$$
\underset{\ve\to 0^+}{\liminf} \int_{\Rn}\int_{\Rn} \frac{|f(x)-f(y)|^p}{|x-y|^p}\rho_\ve(|x-y|) dydx < \infty.
$$
If $p>1$, then $f\in W^{1,p}$ and
\begin{equation}\label{thesisp}
\underset{\ve \to 0^+}{\lim} \int_{\Rn}\int_{\Rn} \frac{|f(x)-f(y)|^p}{|x-y|^p}\rho_\ve(|x-y|) dydx= K_{p,n} \int_{\Rn} |\nabla f(x)|^p dx.
\end{equation}
If instead $p=1$, then $f\in BV$ and
\begin{equation}\label{thesis1}
\underset{\ve \to 0^+}{\lim} \int_{\Rn}\int_{\Rn} \frac{|f(x)-f(y)|}{|x-y|}\rho_\ve(|x-y|) dydx= K_{1,n} \operatorname{Var}(f).
\end{equation}}

In \eqref{thesis1} we have denoted with $\operatorname{Var}(f)$ the total variation of $f$ in the sense of De Giorgi (when $f\in W^{1,1}$ one has $\operatorname{Var}(f) = \int_{\Rn} |\nabla f(x)| dx$). We also remark that for $n\ge 2$ the equality \eqref{thesis1} was proved by D\'avila in \cite{Da}. From Theorem \hyperref[T:bbm]{A} one immediately obtains the limiting behaviour of the seminorms \eqref{ags}. To see this, it is enough for $0<s<1$ to let $\ve=1-s$ and take 
$$
\rho_{1-s}(r)=\begin{cases}
\frac{(1-s)p}{r^{n-(1-s)p}}, \qquad\,\,\,\,\,\, \ 0<r< 1,  \\
0 \qquad\quad\quad\quad\ \ \ \ \,\, \ r\geq 1.
\end{cases}
$$
It is easy to see that \eqref{condbbm} are satisfied and that \eqref{thesisp} gives in such case
\begin{equation}\label{caso1}
\underset{s \to 1^-}{\lim} (1-s)p \int_{\Rn}\int_{\Rn} \frac{|f(x)-f(y)|^p}{|x-y|^{n+sp}} dydx= K_{p,n} ||\nabla f||^p_p.
\end{equation}
From \eqref{caso1}, and from the identity
\begin{equation}\label{Kappa}
K_{p,n}=2\pi^{\frac{n-1}{2}}\frac{\G\left(\frac{p+1}{2}\right)}{\G\left(\frac{n+p}{2}\right)},
\end{equation} 
one concludes that
\begin{equation}\label{seminorm}
\underset{s \to 1^-}{\lim} (1-s)\int_{\Rn}\int_{\Rn} \frac{|f(x)-f(y)|^p}{|x-y|^{n+sp}} dydx= 2\pi^{\frac{n-1}{2}}\frac{\G\left(\frac{p+1}{2}\right)}{p\G\left(\frac{n+p}{2}\right)} ||\nabla f||^p_p.
\end{equation}

To introduce the results in this paper we now emphasise a different perspective on Theorem \hyperref[T:bbm]{A}. If, in fact, we take $\rho_\ve=\rho_{t}$, with
\begin{equation}\label{rho}
\rho_{t}(r)= \frac{\pi^{\frac{n}{2}}}{2^{p-1} \G\left(\frac{n+p}{2}\right)} \frac{r^{p}}{t^{\frac{p}{2}}}\frac{e^{-\frac{r^2}{4t}}}{(4\pi t)^{\frac{n}{2}}},
\end{equation}
then it is easy to see that also such $\rho_t$ satisfies \eqref{condbbm}. Furthermore, with this choice we can write for $1\le p < \infty$
\begin{align*}
& \int_{\Rn}\int_{\Rn} \frac{|f(x)-f(y)|^p}{|x-y|^p}\rho_\ve(|x-y|) dydx = \frac{\pi^{\frac{n}{2}}}{2^{p-1} \G\left(\frac{n+p}{2}\right)} \frac{1}{t^{\frac{p}{2}}}\int_{\Rn} P_t(|f-f(x)|^p)(x) dx, 
\end{align*}
where we have denoted by $P_t f(x) = (4\pi t)^{-\frac{n}{2}}\int_{\Rn} e^{-\frac{|x-y|^2}{4t}} f(y) dy$ the heat semigroup in $\Rn$. If we combine this observation with \eqref{Kappa} and with Legendre duplication formula for the gamma function (see \cite[p.3]{Le}), which gives
$2^{p-1} \G(p/2) \G\left(\frac{p+1}{2}\right) = \sqrt \pi \G(p),
$ we obtain the following notable consequence of Theorem \hyperref[T:bbm]{A}.

\vskip 0.3cm

\noindent \textbf{Theorem B.}\label{C:bbm}\  
\emph{
Assume $1\le p <\infty$. Let $f\in L^p(\Rn)$ and suppose that  
$$
\underset{t\to 0^+}{\liminf} \frac{1}{t^{\frac{p}{2}}}\int_{\Rn} P_t(|f-f(x)|^p)(x) dx < \infty.
$$
If $p>1$, then $f\in W^{1,p}$ and
\begin{equation}\label{thesispPtk}
\underset{t \to 0^+}{\lim} \frac{1}{t^{\frac{p}{2}}}\int_{\Rn} P_t(|f-f(x)|^p)(x) dx = \frac{2 \G(p)}{\G(p/2)} \int_{\Rn} |\nabla f(x)|^p dx.
\end{equation}
If instead $p=1$, then $f\in BV$ and
\begin{equation}\label{thesis11}
\underset{t \to 0^+}{\lim} \frac{1}{\sqrt{t}}\int_{\Rn} P_t(|f-f(x)|)(x) dx= \frac{2}{\sqrt \pi} \operatorname{Var}(f).
\end{equation}}
One remarkable aspect of \eqref{thesispPtk}, \eqref{thesis11} is the dimensionless constant $\frac{2 \G(p)}{\G(p/2)}$ in the right-hand side. 

For the purpose of the present work it is important for the reader to keep in mind that, while we have presented Theorem \hyperref[T:bbm]{B} as a consequence of Theorem \hyperref[T:bbm]{A}, we could have derived the dimensionless heat semigroup characterisations \eqref{thesispPtk}, \eqref{thesis11} of $W^{1,p}$ and $BV$  completely independently of Theorem \hyperref[T:bbm]{A}. In fact, once Theorem \hyperref[T:bbm]{B} is independently proved, one can go full circle and easily obtain from it a dimensionless heat semigroup version of the characterisation  \eqref{seminorm}. Such a perspective, which is close in spirit to M. Ledoux' approach to the isoperimetric inequality in \cite{Led}, represents the starting point of our work, to whose description we now turn.

One of the main objectives of the present paper is to establish, independently of a result such as Theorem \hyperref[T:bbm]{A}, a surprising generalisation of Theorem \hyperref[T:bbm]{B} that we state as Theorems \ref{T:mainp} and \ref{T:p1} below. To provide the reader with a perspective on our results we note that if, as we have done above, one looks at Theorem \hyperref[T:bbm]{B} as a corollary of Theorem \hyperref[T:bbm]{A}, then the spherical symmetry of the approximate identities $\rho_\ve(|x-y|)$, and therefore of the Euclidean heat kernel in \eqref{rho}, seems to play a crucial role in the dimensionless characterisations \eqref{thesispPtk} and \eqref{thesis11}. With this comment in mind, we mention there has been considerable effort in recent years in extending Theorem \hyperref[T:bbm]{A} to various non-Euclidean settings, see \cite{Bar, Lud, CLL, FMPPS, KM, CMSV, Go, CDPP, ArB, HP} for a list, far from being exhaustive, of some of the interesting papers in the subject. In these works the approach is similar to that in the Euclidean setting, and this is reflected in the fact that the relevant approximate identities $\rho_\ve$ either depend on a distance $d(x,y)$, or are asymptotically close in small scales to the well-understood symmetric scenario of $\Rn$.

The point of view of our work is different since, as we have already said, our initial motivation was to understand a result such as Theorem \hyperref[T:bbm]{B} completely independently from Theorem \hyperref[T:bbm]{A}. In this endevor, one immediately runs into the following potentially serious obstruction.

\medskip

\noindent \textbf{Problem:} \emph{Are universal characterisations such as \eqref{thesispPtk} and \eqref{thesis11} even possible in a genuinely non-Riemannian ambient, when the spherical symmetry, or any other symmetries, of the heat kernel are completely lost?} 

\medskip

Concerning this problem a testing ground of basic interest is, for the reasons that we explain below, that of a connected, simply connected Lie group $\bG$ whose Lie algebra admits a stratification $\bg=\bg_1 \oplus \cdots \oplus \bg_r$ which is $r$-nilpotent, i.e.,
$[\bg_1,\bg_j] = \bg_{j+1},$ $j = 1,...,r-1$, $[\bg_j,\bg_r] = \{0\}$, $j =
1,..., r$.
The study of these Lie groups presents considerable challenges and many basic questions pertaining their analytical and geometric properties presently remain fully open. Nowadays known as Carnot groups, they model physical systems with constrained dynamics, in which motion is only possible in a prescribed set of directions in the tangent space (sub-Riemannian, versus Riemannian geometry), see E. Cartan's seminal work \cite{Ca}. Every stratified nilpotent Lie group is endowed with an important second order partial differential operator. The idea goes back to the visionary address of E. Stein \cite{Stein}.  Fix a basis $\{e_1,...,e_{m}\}$ of the Lie algebra generating layer $\bg_1$ (called the horizontal layer) and define left-invariant vector fields on $\bG$ by the rule $X_j(g) = dL_g(e_j)$, $g\in \bG$,
where $dL_g$ is the differential of the left-translation operator $L_g(g') = g \circ g'$. We indicate with $|\nabla_H f|^2 = \sum_{i=1}^m (X_i f)^2$ the horizontal gradient of a function $f$ with respect to the basis $\{e_1,...,e_m\}$. 
Associated with such \emph{carr\'e du champ} there is a natural left-invariant intrinsic distance in $\bG$ defined by 
\begin{equation}\label{d}
d(g,g') \overset{def}{=} \sup \{f(g) - f(g')\mid f\in C^\infty(\bG),\ |\nabla_H f|^2\le 1\}.
\end{equation} 
Such $d(g,g')$ coincides with the Carnot-Carath\'eodory distance, see Gromov's beautiful account \cite{Gro}. We respectively denote by $W^{1,p}(\bG)$ and $BV(\bG)$ the Folland-Stein Sobolev space and the space of $L^1$ functions having bounded variation with respect to the horizontal bundle, see Section \ref{S:prelim} for precise definitions and notations. The horizontal Laplacian relative to $\{e_1,...,e_m\}$ is defined as
\begin{equation}\label{L}
\mathscr L = \sum_{i=1}^m X_i^2.
\end{equation}
When the step of the stratification of $\bg$ is $r=1$, then the group is Abelian and we are back into the familiar Riemannian setting of $\Rn$, in which case $\mathscr L = \Delta$ is the standard Laplacian. However, in the genuinely non-Abelian situation when $r>1$, then the differential operator $\mathscr L$ fails to be elliptic at every point of the ambient space $\bG$, but it possesses nonetheless a heat semigroup $P_t f(g) = e^{-t \mathscr L} f(g) = \int_{\bG} p(g,g',t) f(g') dg'$, see the construction in Folland's work \cite{Fo}. Such semigroup is positive, formally self-adjoint and stochastically complete, i.e. $P_t 1 = 1$. 

The heat kernel $p(g,g',t)$ satisfies appropriate Gaussian estimates with respect to the metric $d(g,g')$ (see Proposition \ref{P:gaussian} below), but this fact is of no help when it comes to a universal statement such as Theorem \hyperref[T:bbm]{B} since, in general, there is no known explicit representation of $p(g,g',t)$, and such heat kernel fails to have any symmetry whatsoever. In particular, it is not a function of the distance $d(g,g')$, nor it is for instance spherically symmetric in any of the layers $\bg_i$, $i=1,...,r$, of the Lie algebra (see the discussion in the opening of Section \ref{S:new}).
Despite these disheartening aspects, we have the following two surprising results.

\begin{theorem}\label{T:mainp}
Let $1<p<\infty$. Then
$$
W^{1,p}(\bG) = \{f\in L^p(\bG)\mid \underset{t\to 0^+}{\liminf}\ \frac{1}{t^{\frac{p}{2}}}\int_{\bG} P_t(|f-f(g)|^p)(g) dg <\infty\}.
$$
Furthermore, if $f\in W^{1,p}(\bG)$ then
\begin{equation}\label{2p}
\underset{t \to 0^+}{\lim} \frac{1}{t^{\frac{p}{2}}}\int_{\bG} P_t(|f-f(g)|^p)(g) dg = \frac{2 \G(p)}{\G(p/2)} \int_{\bG} |\nabla_H f(g)|^p dg.
\end{equation}
\end{theorem} 

Concerning the case $p=1$, the following is our second main result.

\begin{theorem}\label{T:p1}
We have
\begin{equation}\label{1uno}
BV(\bG) =\left\{f\in L^1(\bG)\mid \underset{t \to 0^+}{\liminf}\  \frac{1}{\sqrt t} \int_{\bG} P_t\left(|f - f(g)|\right)(g) dg<\infty \right\},
\end{equation}
and for any $f\in W^{1,1}(\bG)$
\begin{equation}\label{2unouno}
\underset{t \to 0^+}{\lim}  \frac{1}{\sqrt{t}}\ \int_{\bG} P_t\left(|f - f(g)|\right)(g) dg = \frac{2}{\sqrt{\pi}} \int_{\bG} |\nabla_H f(g)| dg.
\end{equation}
Furthermore, if the Carnot group $\bG$ has the property \emph{(B)}\footnote{for this property the reader should see Definition \ref{D:B} below}, then for any $f\in BV(\bG)$ we have
\begin{equation}\label{2uno}
\underset{t \to 0^+}{\lim}  \frac{1}{\sqrt{t}}\ \int_{\bG} P_t\left(|f - f(g)|\right)(g) dg = \frac{2}{\sqrt{\pi}} {\rm{Var}}_\bG(f).
\end{equation}
\end{theorem}

We draw the reader's attention to the remarkable similarity between \eqref{2p}, \eqref{2uno} and their Euclidean predecessors \eqref{thesispPtk}, \eqref{thesis11}. The presence of the universal constant $\frac{2 \G(p)}{\G(p/2)}$ in the right-hand sides of \eqref{2p}, \eqref{2uno} underscores a remarkable general character of the heat semigroup that we next clarify. Having stated our main results, we must explain our comment on their surprising aspect. While we refer the reader to Section \ref{S:new} for a detailed discussion of this point, here we confine ourselves to mention that the crucial novelty in our approach is Theorem \ref{T:int} below. The latter represents an \emph{integral decoupling property} of the sub-Riemannian heat kernels. With such result in hands we obtain the basic Lemma \ref{L:id}. It is precisely this lemma that accounts for the universal character of Theorems \ref{T:mainp} and \ref{T:p1}. We mention that Lemma \ref{L:id} is reminiscent of two remarkable properties of the classical heat semigroup first discovered respectively by Ledoux in his approach to the isoperimetric inequality \cite{Led}, and by Huisken in his work on singularities of flow by mean curvature \cite{Hui}. It is worth remarking at this point that, as we explain in Section \ref{SS:fulvio} below, some experts in the noncommutative analysis community are familiar with the integral decoupling property in Theorem \ref{T:int}. However, the use that we make of such result is completely new. In this respect, we mention that the special case of Carnot groups of step 2 in Theorem \ref{T:p1} was treated in our recent work \cite{GTbbmd}. In that setting we were able to extract the crucial information \eqref{punoint} in Lemma \ref{L:id} from the explicit Gaveau-Hulanicki-Cygan representation formula \eqref{ournucleo} below. No such formula is available for Carnot groups of step 3 or higher, and it is precisely a result such as Theorem \ref{T:int} that allows to successfully handle this situation.

As previously mentioned, in the special situation when $\bG=\Rn$ we recover Theorem \hyperref[T:bbm]{B} from Theorems \ref{T:mainp} and \ref{T:p1}, as well as a dimensionless heat semigroup formulation of the Brezis-Bourgain-Mironescu limiting behaviour \eqref{seminorm}. We next show that this comment extends to the geometric setting of the present paper. We begin by introducing the relevant function spaces.

\begin{definition}\label{D:besov} Let $\bG$ be a Carnot group. For any $0<s<1$ and $1\le p<\infty$ we define
the \emph{fractional Sobolev space} $\Bps$ as the collection of all functions $f\in L^p(\bG)$ such that the seminorm
$$
\mathscr N_{s,p}(f) = \left(\int_0^\infty  \frac{1}{t^{\frac{s p}2 +1}} \int_{\bG} P_t\left(|f - f(g)|^p\right)(g) dg dt\right)^{\frac 1p} < \infty.
$$
\end{definition}
The norm 
\[
||f||_{\Bps} = ||f||_{\Lp(\bG)} + \mathscr N_{s,p}(f)
\]
 turns $\Bps$ into a Banach space. We stress that the space $\Bps$ is nontrivial since, for instance, it contains $W^{1,p}(\bG)$ (see Lemma \ref{L:inclus} below). We also emphasise that, when the step $r=1$ and $\bG\cong \R^n$ is Abelian, then the space $\Bps$ coincides with the classical Aronszajn-Gagliardo-Slobedetzky space of fractional order $W^{s,p}(\R^n)$ of the functions $f\in L^p$ with finite seminorm $[f]^p_{s,p}$ in \eqref{ags}. It is in fact an exercise to recognise in this case that
\[
\mathscr N_{s,p}(f)^p = \frac{2^{sp} \G(\frac{n+sp}2)}{\pi^{\frac n2}}\ [f]_{s,p}^p.
\]
Concerning the spaces $\Bps$ our main result is the following. It provides a sub-Riemannian dimensionless version of the above mentioned limiting phenomenon \eqref{seminorm}.

\begin{theorem}\label{T:bbmG}
Let $\bG$ be a Carnot group. Then
\begin{equation}\label{1sp}
W^{1,p}(\bG) = \{f\in L^p(\bG)\mid \underset{s\to 1^-}{\liminf}\ (1-s) \mathscr N_{s,p}(f)^p <\infty\}\qquad \mbox{ for }1< p<\infty,
\end{equation}
and
\begin{equation}\label{1suno}
BV(\bG) =\left\{f\in L^1(\bG)\mid \underset{s\to 1^-}{\liminf}\ (1-s) \mathscr N_{s,1}(f) <\infty \right\}.
\end{equation}
For any $1\leq p<\infty $ and $f\in W^{1,p}(\bG)$, one has 
\begin{equation}\label{2sp}
\underset{s\to 1^-}{\lim}\ (1-s) \mathscr N_{s,p}(f)^p = \frac{4 \G(p)}{p\G(p/2)} \int_{\bG} |\nabla_H f(g)|^p dg.
\end{equation}
Furthermore, if the Carnot group $\bG$ has the property \emph{(B)}, then for any $f\in BV(\bG)$ we have
\begin{equation}\label{2suno}
\underset{s\to 1^-}{\lim}\ (1-s) \mathscr N_{s,1}(f) = \frac{4}{\sqrt{\pi}}  {\rm{Var}}_\bG(f).
\end{equation}
\end{theorem}

Our last result concerns the asymptotic behaviour in $s$ of the seminorms $\mathscr N_{s,p}(f)$ at the other end-point of interval $(0,1)$. Such result provides a dimensionless generalisation of that proved by Maz'ya and Shaposhnikova in \cite{MS}.

\begin{theorem}\label{T:MS}
Let $\bG$ be a Carnot group, and $1\leq p <\infty$. Suppose that $f\in \underset{0<s<1}{\bigcup}\Bps$. Then,
$$
\underset{s\to 0^+}{\lim} s \mathscr N_{s,p}(f)^p = \frac{4}{p} ||f||_p^p.
$$
\end{theorem}

In closing, we briefly discuss the structure of the paper. In Section \ref{S:prelim} we recall the geometric setup. In Section \ref{S:prep} we present some basic preparatory results that will be needed in the rest of the paper. 
Section \ref{S:new} is central to the rest of our work, but we also feel that it has an independent interest with consequences that go beyond those in the present work. We establish Theorem \ref{T:int} which, as we have said, represents a key property of the heat kernel in a Carnot group. With such result in hand, we obtain the crucial Lemma \ref{L:id}. Section \ref{S:proofs}
is devoted to proving Theorems \ref{T:mainp} and \ref{T:p1}. Finally, in Section \ref{S:seminorms} we prove Theorems \ref{T:bbmG} and \ref{T:MS}.



\section{Background}\label{S:prelim}

In the last decades various aspects of analysis and geometry in Carnot groups have attracted a lot of attention, and we refer to the monographs \cite{FS, V, CG, VSC, Gro, Ri, BLU, Gparis} for insightful perspectives. For the reader's convenience we have collected in this section the background material which is needed in the present paper. Although the relevant geometric setting has been introduced in Section \ref{S:intro}, we recall it here.

\begin{definition}\label{D:carnot}
Given $r\in \mathbb N$, a \emph{Carnot group} of step $r$ is a simply-connected real Lie group $(\bG, \circ)$ whose Lie algebra $\bg$ is stratified and $r$-nilpotent. This means that there exist vector spaces $\bg_1,...,\bg_r$ such that  
\begin{itemize}
\item[(i)] $\bg=\bg_1\oplus \dots\oplus\bg_r$;
\item[(ii)] $[\bg_1,\bg_j] = \bg_{j+1}$, $j=1,...,r-1,\ \ \ [\bg_1,\bg_r] = \{0\}$.
\end{itemize}
\end{definition}
We assume that $\bg$ is endowed with a
scalar product $\langle\cdot,\cdot\rangle$ with respect to which the layers $\bg_j's$, $j=1,...,r$, are mutually orthogonal.  We let $m_j =$ dim\ $\bg_j$, $j=
1,...,r$, and denote by $N = m_1 + ... + m_r$ the topological
dimension of $\bG$.  From the assumption (ii) on the Lie
algebra it is clear that any basis of the first layer $\bg_1$ bracket generates the whole Lie algebra $\bg$. Because of such special role $\bg_1$ is usually called the horizontal
layer of the stratification. For ease of notation we henceforth write $m = m_1$. In the case in which $r =1$ we are in the Abelian situation in which $\bg = \bg_1$, and thus $\bG$ is isomorphic to $\R^m$, where $m =$ dim\ $\bg_1$. We are thus back in $\R^m$, there is no sub-Riemannian geometry involved and everything is classical. We are of course primarily interested in the genuinely non-Riemannian setting $r>1$. The exponential map $\exp : \bg \to \bG$ defines an analytic
diffeomorphism of the Lie algebra $\bg$ onto $\bG$, see e.g. \cite[Sec. 2.10 forward]{V}. Using such diffeomorphism, whenever convenient we will routinely identify a point $g = \exp \xi \in \bG$ with its logarithmic image $\xi = \exp^{-1} g\in \bg$. With such identification, if $\xi = \xi_1+...+\xi_r$, we let $\xi_1 = z_{1} e_{1} +...+z_{m}e_{m}\in \bg_1$, and  $\xi_j = \sigma_{j,1} e_{j,1} +...+\sigma_{j,m_j}e_{j,m_j}\in \bg_j$, $j = 1,...,r$. Whenever convenient, see for instance the important expression \eqref{Xi} below, we will routinely identify the vector $\xi_1\in \bg_1$ with the point $z = (z_1,...,z_m)\in \Rm$, and the vector $\xi_2+...+\xi_r$ with the point $\sigma = (\sigma_{2},...,\sigma_{r})\in \R^{N-m}$, where 
$\sigma_j = (\sigma_{j,1},...,\sigma_{j,m_j}) \in \mathbb R^{m_j}$. 
Given $\xi, \eta\in \bg$, the Baker-Campbell-Hausdorff formula reads
\begin{equation}\label{BCH}
\exp(\xi) \circ \exp(\eta) = \exp{\bigg(\xi + \eta + \frac{1}{2}
[\xi,\eta] + \frac{1}{12} \big\{[\xi,[\xi,\eta]] -
[\eta,[\xi,\eta]]\big\} + ...\bigg)},
\end{equation}
where the dots indicate commutators of order four and higher, see \cite[Sec. 2.15]{V}.  Furthermore, since by (ii) in Definition \ref{D:carnot} all commutators of order higher than $r$ are trivial, in every Carnot group the Baker-Campbell-Hausdorff series in the right-hand side of \eqref{BCH} is finite. 
Using \eqref{BCH}, with $g = \exp \xi, g' = \exp \xi'$, one can recover the group law $g \circ g'$ in $\bG$ from the knowledge of the algebraic commutation relations between the elements of its Lie algebra. We respectively denote
by $L_g(g') = g \circ g'$ and $R_g(g') = g'\circ g$ the left- and right-translation operator by an element $g\in
\bG$. We indicate by $dg$ the bi-invariant
Haar measure on $\bG$ obtained by lifting via the exponential map
 the Lebesgue measure on $\bg$. 

The stratification (ii) induces in $\bg$ a natural one-parameter family of non-isotropic dilations by assigning to each
element of the layer $\bg_j$ the formal degree $j$.
Accordingly, if $\xi = \xi_1 + ... + \xi_r \in \bg$, with $\xi_j\in \bg_j$,
one defines dilations on $\bg$ by the rule
$\Delta_\lambda \xi = \lambda \xi_1 + ... + \lambda^r
\xi_r,$
and then use the exponential map to transfer such
anisotropic dilations to the group $\bG$ as
follows
\begin{equation}\label{dilG}
\delta_\lambda(g) = \exp \circ \Delta_\lambda \circ
\exp^{-1} g.
\end{equation}
The homogeneous dimension of $\bG$ with respect to \eqref{dilG} is the number $Q = \sum_{j=1}^r j m_j.$ Such number plays an important role in the analysis of Carnot groups. 
The motivation for this name comes from the equation 
$$(d\circ\delta_\lambda)(g) = \lambda^Q dg.$$
In the
non-Abelian case $r>1$, one clearly has $Q>N$. We will use the non-isotropic gauge in $\bg$ defined in the following way $|\xi| = \left(\sum_{j=1}^r ||\xi_j||^{2r!/j}\right)^{1/2r!}$, see \cite{Fo}. It is obvious that $|\Delta_\lambda \xi| = \lambda |\xi|$ for $\la>0$.
One defines a non-isotropic gauge in the group $\bG$ by letting $|g| = |\xi|$ for $g = \exp \xi$. Clearly, $|\cdot|\in C^\infty(\bG\setminus\{e\})$, and moreover $|\delta_\la g| = \la |g|$ for every $g\in \bG$ and $\la>0$. The pseudodistance generated by such gauge is equivalent to the intrinsic distance \eqref{d} on $\bG$, i.e., there exists a universal constant $c_1>0$ such that for every $g, g'\in \bG$
$$c_1 |(g')^{-1} \circ g| \le d(g,g') \le c_1^{-1} |(g')^{-1} \circ g|,$$  
see \cite{Fo}. 
Given a orthonormal basis $\{e_1,...,e_m\}$ of the horizontal layer $\bg_1$ one associates corresponding left-invariant $C^\infty$ vector fields on $\bG$ by the formula $X_i(g) = (L_g)_\star(e_i)$, $i=1,...,m$, where $(L_g)_\star$ indicates the differential of $L_g$. We note explicitly that, given a smooth function $u$ on $\bG$, the derivative of $u$ in $g\in \bG$ along the vector field $X_i$ is given by the Lie formula
\begin{equation}\label{lie}
X_i u(g)  = \frac{d}{ds} u(g \exp s e_i)\big|_{s=0}.
\end{equation}

\subsection{Stratified mean-value formula}\label{SS:stratified} Using the Baker-Campbell-Hausdorff formula \eqref{BCH} we can express \eqref{lie} in the logarithmic coordinates, obtaining the following representation that will be useful subsequently in this paper, see \cite[Prop. (1.26)]{FS} or also \cite[Remark 1.4.6]{BLU}. As previously mentioned, the layer $\bg_j$, $j=1,...,r,$ in the stratification of $\bg$ is assigned the formal degree $j$. Correspondingly, each homogeneous monomial $\xi_1^{\alpha_1}
\xi_2^{\alpha_2}...\xi_r^{\alpha_r}$, with multi-indices $\alpha_j =
(\alpha_{j,1},...,\alpha_{j,m_j}),\ j=1,...,r,$ is said to have
\emph{weighted degree} $k$ if
\[
\sum_{j=1}^r j (\sum_{s=1}^{m_j} \alpha_{j,s}) = k.
\]
Then for each $i = 1,...,m$ we have
\begin{align}\label{Xi}
X_i & = \frac{\partial }{\partial{z_i}} +
\sum_{j=2}^{r}\sum_{s=1}^{m_j}
b^s_{j,i}(z_1,...,\sigma_{{j-1},m_{(j-1)}}) \frac{\partial
}{\partial{\sigma_{j,s}}}
\\
& = \frac{\partial }{\partial{z_i}} +
\sum_{j=2}^{r}\sum_{s=1}^{m_j} b^s_{j,i}(\xi_1,...,\xi_{j-1})
\frac{\partial }{\partial{\sigma_{j,s}}}, \notag
\end{align}
where each $b^s_{j,i}$ is a homogeneous polynomial of
weighted degree $j-1$. Since we assume that $\bG$ is endowed with a left-invariant Riemannian metric with respect to which the vector fields $\{X_1,...,X_m\}$ are orthonormal, then given a smooth function $u$ on $\bG$ we denote by 
$\nh u = \sum_{i=1}^m X_i u X_i$
its horizontal gradient, and denote $|\nh u|^2 = \sum_{i=1}^m (X_i u)^2$. The quasi-metric open ball centered at $g$ and with radius $r>0$ with respect to the non-isotropic gauge $|\cdot|$ will be denoted by $B(g,r)$. We will need the following special case of the stratified Taylor inequality, see \cite[Theor. 1.42]{FS}. 

\begin{proposition}\label{P:taylor}
Let $f\in C^1(\bG)$. There exist universal constants $C, b>0$ such that for every $g\in \bG$ and $r>0$ one has for $g'\in B(g,r)$
\[
|f(g') - f(g) - \langle\nh f(g),z'-z\rangle| \leq C r \underset{g''\in B(g, b r)}{\sup} |\nh f(g'') - \nh f(g)|.
\]
\end{proposition}

\subsection{The heat kernel}\label{SS:heat} The horizontal Laplacian $\mathscr L$ relative to $\{e_1,...,e_m\}$ is defined as in \eqref{L}, and we denote by $P_t f(g) = e^{-t \mathscr L} f(g) = \int_{\bG} p(g,g',t) f(g') dg'$ the corresponding heat semigroup constructed by Folland in \cite{Fo}. Since by \eqref{Xi} we have $X_i^\star = - X_i$, $i=1,...,m$, the heat kernel is symmetric $p(g,g',t)=p(g',g,t)$. Furthermore, it satisfies the following properties. Hereafter, we denote with $e\in \bG$ the identity element.

\begin{proposition}\label{P:prop}
For every $g, g', g''\in \bG$ and $t>0$, one has
\begin{itemize}
\item[(i)] $p(g,g',t)=p(g''\circ g,g''\circ g',t)$;
\item[(ii)] $p(g,e,t)=t^{-\frac{Q}{2}}p(\delta_{1/\sqrt{t}}g,e,1)$;
\item[(iii)] $P_t 1(g) = \int_\bG p(g,g',t) dg'=1$.
\end{itemize}
\end{proposition}
In addition, one has the following basic Gaussian estimates, see \cite{JS, VSC}. Such estimates play a ubiquitous role in the present work. 

\begin{proposition}\label{P:gaussian}
There exist universal constants $\alpha, \beta>0$ and $C>1$ such that for every $g, g' \in \bG$, $t > 0$, and $j\in\{1,\ldots,m\}$
\begin{equation}\label{gauss0}
\frac{C^{-1}}{t^{\frac Q2}} \exp \bigg(-\alpha\frac{|(g')^{-1}\circ g|^2}{t}\bigg)\leq p(g,g',t) \leq \frac{C}{t^{\frac Q2 }} \exp \bigg(-\beta\frac{ |(g')^{-1}\circ g|^2}{t}\bigg),
\end{equation}
\begin{equation}\label{gauss1}
\left|X_{j}p(g,g',t)\right|\ \leq\ \frac{C}{t^{\frac{Q+1}{2}}} \exp \bigg(-\beta\frac{ |(g')^{-1}\circ g|^2}{t}\bigg),
\end{equation}
\begin{equation}\label{gauss2}
\left|X^2_{j}p(g,g',t)\right| + \left|\partial_t p(g,g',t)\right|\ \leq\  \frac{C}{t^{\frac Q2 +1}} \exp \bigg(-\beta\frac{ |(g')^{-1}\circ g|^2}{t}\bigg).
\end{equation}
\end{proposition}

We next introduce the relevant functional spaces for the present work. If $1\le p<\infty$, the Folland-Stein Sobolev space of order one is $W^{1,p}(\bG) = \{f\in L^p(\bG)\mid X_i f\in L^p(\bG), i=1,...,m\}$. Endowed with the norm 
\[
||f||_{W^{1,p}(\bG)} = ||f||_{L^p(\bG)} + ||\nh f||_{L^p(\bG)}
\]
this is a Banach space, which is reflexive when $p>1$. Such latter property will be used in Lemma \ref{L:inf} below. We will need the following approximation property
\begin{equation}\label{dense}
C^\infty_0(\bG)\ \text{is dense in}\ W^{1,p}(\bG)\quad\ \ \ \ \ 1\le p<\infty,
\end{equation}
see \cite[Theor. 4.5]{Fo} and also \cite[Theor. A.2]{GN}. 

\subsection{Bounded variation and coarea}\label{SS:bv}
In connection with Theorem \ref{T:p1} we need to recall the space of functions with horizontal bounded variation introduced in \cite{CDG}. Let $\mathscr F = \{\zeta = (\zeta_1,...,\zeta_m)\in C^1_0(\bG,\Rm)\mid ||\zeta||_\infty = \underset{g\in \bG}{\sup} (\sum_{i=1}^m \zeta_i(g)^2)^{1/2} \le 1\}$. Then $BV(\bG) = \{f\in L^1(\bG)\mid \operatorname{Var}_\bG(f)<\infty\}$, where
\begin{equation}\label{var}
\operatorname{Var}_\bG(f) = \underset{\zeta\in \mathscr F}{\sup} \int_{\bG} f \sum_{i=1}^m X_i \zeta_i  dg.
\end{equation}
Endowed with the norm $||f||_{L^1(\bG)} + \operatorname{Var}_\bG(f)$ the space $BV(\bG)$ is a Banach space. When $f\in W^{1,1}(\bG)$, then $f\in BV(\bG)$ and one has
$\operatorname{Var}_\bG(f) = \int_{\bG} |\nh f| dg$,
but the inclusion is strict. In some formulas we indicate with $d\operatorname{Var}_\bG(f)$ the differential of horizontal total variation associated with a function $f\in BV(\bG)$, so that $\operatorname{Var}_\bG(f) = \int_\bG d\operatorname{Var}_\bG(f)(g)$. It is well-known that there exists a $\operatorname{Var}_\bG(f)$-measurable function $\sigma_f:\bG\to \Rm$, such that $||\sigma_f(g)|| = 1$ for $\operatorname{Var}_\bG(f)$-a.e. $g\in \bG$, and for which one has for any $\zeta\in \mathscr F$
\begin{equation}\label{sigmaf}
\int_{\bG} f \sum_{i=1}^m X_i \zeta_i  dg = \int_\bG \langle\sigma_f(g),\zeta(g)\rangle d\operatorname{Var}_\bG(f)(g).
\end{equation}

A measurable set $E\subset \bG$ is said to have finite $\bG$-perimeter if $1_E\in BV(\bG)$, and in this case we denote $P_\bG(E) = \operatorname{Var}_\bG(1_E)$. The following two basic results are special cases of respectively \cite[Theor. 1.14]{GN} and \cite[Theor. 5.2]{GN}.

\begin{theorem}\label{T:dense}
Given $f\in BV(\bG)$, there exists $f_k\in C^\infty_0(\bG)$ such that as $k\to \infty$ one has
$f_k \longrightarrow f$ in $L^1(\bG)$, and $\int_{\bG} |\nh f_k| dg \longrightarrow \operatorname{Var}_\bG(f)$.
\end{theorem}

\begin{theorem}\label{T:coarea}
Let $f\in BV(\bG)$. Then $P_\bG(E^f_\tau)<\infty$ for a.e. $\tau \in \R$ and
\begin{equation}\label{coarea}
\operatorname{Var}_\bG(f) = \int_\R P_\bG(E^f_\tau) d\tau,
\end{equation}
where we have let
$E^f_\tau=\{g\in\bG\mid f(g)>\tau\}$.  Vice-versa, if $f\in L^1(\bG)$ and the right-hand side of \eqref{coarea} is finite, then $f\in BV(\bG)$.
\end{theorem}

Let $E\subset \bG$ be a set having finite perimeter. We denote by $\p^\star E$ its reduced boundary, see \cite[Def. 2.8]{BMP}. Given a point $g\in \p^\star E$ we indicate by $\nu_E(g)$ the measure theoretic horizontal normal at $g$. Also, given a point $g_0\in \bG$ we let $E_{g_0,r} = \delta_{1/r} L_{g_0^{-1}}(E)$. The vertical planes and half-spaces associated with a unit vector $\nu\in \Rm$ are respectively defined by
\[
T_\bG(\nu) = \{(z,\sigma)\in \bG\mid \langle z,\nu\rangle = 0\},\ \ \ \ S_\bG^+(\nu) = \{(z,\sigma)\in \bG\mid \langle z,\nu\rangle \ge 0\}.
\]  
We denote by $T_\nu$ the $(m-1)$-dimensional subspace of $\R^m$ defined as $\left(\rm{span}\{\nu\}\right)^{\perp}$, and indicate with $P_\nu:\Rm\to \Rm$ the orthogonal projection onto $T_\nu$, i.e. $P_\nu z=z-<z,\nu>\nu$. Clearly, its range is  $R(P_\nu) =T_\nu$, and we have $P_\nu^2=P_\nu=P_\nu^*$. One has
\begin{equation}\label{Tnu}
T_\bG(\nu)=T_\nu\times\R^{N-m}=\{(\hat z,\sigma)\in \bG\mid \hat z\in\R^{m},\ \sigma\in \R^{N-m}, \mbox{ such that } P_\nu \hat z=\hat z\}.
\end{equation}

\begin{definition}\label{D:B}
We say that a Carnot group $\bG$ satisfies the property \emph{(B)} if for every set of finite perimeter $E\subset \bG$, and for every $g_0\in \p^\star E$, one has in $L^1_{loc}(\bG)$
\[
1_{E_{g_0,r}}\ \longrightarrow\ 1_{S_\bG^+(\nu_E(g_0))}\quad\mbox{ as }r\to 0^+.
\]
\end{definition}

It is presently not known whether every Carnot group satisfies the property (B), but in \cite[Theor. 3.1]{FSS} it was shown that groups of step two possess this property. The paper \cite{Ma} contains a class of groups, which includes Carnot groups of step two, that satisfy the property (B) (see in this respect \cite[Theor. 4.12]{Ma}). In the proof of Theorem \ref{T:p1} we will need the following result, which is \cite[Theor. 2.14]{BMP}. We stress that although in that paper the authors assume that $\bG$ has step 2, as they themselves emphasise their result continue to hold in any Carnot group with the property (B). 
 
\begin{theorem}\label{T:bmpB}
Let $\bG$ be a Carnot group satisfying property \emph{(B)}. If $f\in BV(\bG)$ one has
\[
\underset{t\to 0^+}{\lim} \frac{1}{\sqrt t} \int_{\bG} P_t(|f - f(g)|)(g) dg = 4 \int_{\bG} \int_{T_\bG(\sigma_f(g))} p(g',e,1) dg' d \operatorname{Var}_\bG(f)(g),
\]
where $\sigma_f$ is the function defined by \eqref{sigmaf} above.
\end{theorem}


\section{Preparatory results}\label{S:prep}

This section contains some preparatory results that will be used in the main body of the paper. 

\subsection{A basic preliminary estimate}\label{SS:basic} We begin with a useful consequence of the property \eqref{dense} and of the upper Gaussian estimate in \eqref{gauss0}.

\begin{lemma}\label{L:diffquot}
Let $1\le p<\infty$. There exists a universal constant $C_p>0$ such that, for all $f\in W^{1,p}(\bG)$ and $t>0$, one has
\begin{equation}\label{suppose}
 t^{-\frac{p}{2}} \int_{\bG} P_t\left(|f - f(g)|^p\right)(g) dg \le C_p \int_{\bG} |\nh f(g)|^p dg.
\end{equation}
\end{lemma} 

\begin{proof}
Suppose first that $f\in C^\infty_0(\bG)$. For any given $g, g'\in \bG$ we write
\[
|f(g\circ g')- f(g)| = \left|\int_0^1 \frac{d}{d\la} f(g\circ \delta_\la g') d\la\right|,
\]
where $\delta_\la$ are the group anisotropic dilations \eqref{dilG}. Recalling the notion of Pansu differential $Df(g)(g')$ in \cite{P}, when the target Carnot group is $\bG' = (\R,+)$ one has
\[
Df(g)(g') = \lim_{\lambda\to 0^+} \frac{f(g \circ \delta_\lambda g') - f(g)}{\lambda} = \frac{d}{d\la} f(g \circ \delta_\lambda g')\big|_{\la=0},
\]
with the limit being locally uniform in $g'\in \bG$. By the Baker-Campbell-Hausdorff formula one finds that when $f$ is smooth then
\[
Df(g)(g') = \langle\nh f(g),z' \rangle.
\]
Substituting in the above equation, integrating in $g\in \bG$, and using the invariance of Lebesgue measure on $\bG$ with respect to right-translations, after changing variable $g\to g'' = g \circ \delta_\lambda g'$, we easily find
\begin{equation}\label{mezzo}
\int_{\bG} |f(g\circ g')- f(g)|^p dg \le ||z'||^p \int_{\bG} |\nh f(g'')|^p dg''.
\end{equation}
Multiplying both sides of \eqref{mezzo} by $p(e,g',t)$, and integrating with respect to $g'$, we obtain
\begin{align*}
&\int_{\bG} p(e,g',t)||z'||^p dg' \int_{\bG} |\nh f(g'')|^p dg''\\ 
&\geq \int_{\bG}\int_{\bG} p(e,g',t)|f(g\circ g')- f(g)|^p dg dg' =\int_{\bG}\int_{\bG} p(e,g^{-1}\circ g''',t)|f(g''')- f(g)|^p dg''' dg \\
&=\int_{\bG}\int_{\bG} p(g,g''',t)|f(g''')- f(g)|^p dg''' dg,
\end{align*}
where in the last equality we have used (i) in Proposition \ref{P:prop}. Exploiting also (ii) in Proposition \ref{P:prop} and \eqref{gauss0}, we find
$$
t^{-\frac{p}{2}}\int_{\bG} p(e,g',t)||z'||^p dg'= \int_{\bG} p(g,e,1)||z||^p dg \leq C \int_{\bG} e^{-\beta |g|^2 }||z||^p dg =:C_p <\infty.
$$
The previous two inequalities yield
\begin{align*}
t^{-\frac{p}{2}} \int_{\bG} P_t\left(|f - f(g)|^p\right)(g) dg &= t^{-\frac{p}{2}} \int_{\bG}\int_{\bG} p(g,g''',t)|f(g''')- f(g)|^p dg''' dg\\
&\leq C_p \int_{\bG} |\nh f(g'')|^p dg''.
\end{align*}
This shows \eqref{suppose} for $f\in C^\infty_0(\bG)$. Appealing to the density result \eqref{dense} it is easy to see that \eqref{suppose} continues to be valid for functions $f\in W^{1,p}(\bG)$.

\end{proof}

\subsection{Heat semigroup estimates for BV functions}\label{SS:heatBV} The following basic result combines the case $p=1$ of Lemma \ref{L:diffquot}, with Theorems \ref{T:dense} and \ref{T:coarea}. It displays the typical character of the geometric Sobolev embedding. Whatever constant works for \eqref{minchietto} below, that same constant works for \eqref{minchiettof}, and vice-versa.

\begin{proposition}\label{P:minchietto}
There exists a universal constant $C_1>0$ such that
if $E\subset \bG$ has finite $\bG$-perimeter, then for every $t>0$ one has\ \footnote{A different proof of \eqref{minchietto} is given in \cite[Prop. 4.2]{BMP}.}
\begin{equation}\label{minchietto}
\frac{1}{\sqrt t} ||P_t\In - \In||_{L^1(\bG)}  \le C_1\ P_\bG(E).
\end{equation}
Furthermore, if $f\in BV(\bG)$, then for every $t>0$ one has
\begin{equation}\label{minchiettof}
\frac 1{\sqrt t}\ \int_{\bG} P_t\left(|f - f(g)|\right)(g) dg \le C_1\ \operatorname{Var}_\bG(f).
\end{equation}
\end{proposition}
\begin{proof}
In view of Theorem \ref{T:dense} there exists a sequence $f_k\in C^\infty_0(\bG)$ such that as $k\to \infty$ one has
$f_k \longrightarrow 1_E$ in $L^1(\bG)$, and $\int_{\bG} |\nh f_k| dg \longrightarrow \operatorname{Var}_\bG(1_E)= P_\bG(E)$. By the case $p=1$ of Lemma \ref{L:diffquot}
one has
\[
\frac{1}{\sqrt t} \int_{\bG} P_t\left(|f_k - f_k(g)|\right)(g) dg \le C_1 \int_{\bG} |\nh f_k(g)| dg.
\]
Since, by possibly passing to a subsequence, we can always assume that $f_k\to 1_E$ a.e. in $\bG$, we have $P_t(|f_k - f_k(g)|)(g)\longrightarrow P_t(|1_E - 1_E(g)|)(g)$ for a.e. $g\in \bG$. This observation and the theorem of Fatou give
\begin{align*}
& \frac{1}{\sqrt t} \int_{\bG}  P_t(|1_E - 1_E(g)|)(g) dg \le \underset{k\to \infty}{\liminf} \frac{1}{\sqrt t}\int_{\bG} P_t\left(|f_k - f_k(g)|\right)(g) dg 
\\
& \le C_1\ \underset{k\to \infty}{\liminf} \int_{\bG} |\nh f_k(g)| dg dg
= C_1\ P_\bG(E).
\end{align*}
To finish the proof of \eqref{minchietto}, all we are left with is observing that from $P_t1=1$ (see (iii) in Proposition \ref{P:prop}) one has
\begin{equation}\label{unoE}
||P_t\In - \In||_{L^1(\bG)} =  \int_{\bG}{P_t(|\In-\In(g)|)(g)dg}.
\end{equation}
To prove \eqref{minchiettof}, we observe that if $f\in L^1(\bG)$, then for a.e. $g, g'\in \bG$ we trivially have
\begin{equation}\label{riesz}
|f(g')-f(g)|=\int_\R \left|1_{E^f_\tau} (g')-1_{E^f_\tau} (g)\right| d\tau.
\end{equation}
In view of \eqref{riesz}, Fubini's theorem allows to write
\begin{equation}\label{slices}
\int_{\bG} P_t(|f-f(g)|)(g) dg = \int_\R \int_{\bG} P_t(|1_{E^f_\tau}-1_{E^f_\tau}(g)|)(g) dg d\tau.
\end{equation}
We now note that if $f\in BV(\bG)$, then by Theorem \ref{T:coarea}
we know that $P_\bG(E^f_\tau)<\infty$ for a.e. $\tau \in \R$. With this in mind, if we combine \eqref{slices} with \eqref{unoE}, we obtain
\begin{align*}
& \frac{1}{\sqrt t} \int_{\bG} P_t(|f - f(g)|)(g) dg = \frac{1}{\sqrt t} \int_\R \int_{\bG} P_t(|1_{E^f_\tau}-1_{E^f_\tau}(g)|)(g) dg d\tau
\\
& = \frac{1}{\sqrt t} \int_\R  ||P_t1_{E^f_\tau} - 1_{E^f_\tau}||_{L^1(\bG)} d\tau \le C_1\ \int_\R P_\bG(E^f_\tau) d\tau = C_1\ \operatorname{Var}_\bG(f),
\end{align*} 
where in the last two passages we have first used \eqref{minchietto}, and then \eqref{coarea} in Theorem \ref{T:coarea}. The latter estimate establishes \eqref{minchiettof}.

\end{proof}

\subsection{Fractional Sobolev spaces and heat semigroup}\label{SS:fracheat}
In this subsection we establish some basic properties of the fractional Sobolev spaces $\Bps$ introduced in Definition \ref{D:besov}. We refer the interested reader to \cite{GTfi, GTiso, BGT} for a different but related setting.

\begin{lemma}\label{L:inclus}
Fix $1\leq p<\infty$. For any $f\in L^p$ we have
\begin{equation}\label{inbesov}
\underset{t \to 0^+}{\limsup}\ t^{-\frac{p}{2}}\ \int_{\bG} P_t\left(|f - f(g)|^p\right)(g) dg <\infty \quad \Longrightarrow \quad f\in \Bps \mbox{ for every}\ s\in (0,1).
\end{equation}
Moreover the following holds
\begin{equation}\label{chain}
W^{1,p}(\bG) \subset \Bps \subset\mathfrak B_{\sigma,p}(\bG)\,\,\,  \mbox{ for every }\,0<\sigma \leq s <1,
\end{equation}
where the inclusions denote continuous embeddings with respect to the relevant topologies.
\end{lemma}
\begin{proof}
Fix $1\leq p<\infty$ and $0<s<1$, and consider any $f\in L^p$ and $\ve>0$. Using $|a-b|^p\leq 2^{p-1} (a^p+b^p)$ together with $P_t1=P^*_t 1=1$, we have
\begin{equation}\label{Ndopo1}
\int_{\ve}^\infty \frac{1}{t^{\frac{s p}2 +1}} \int_{\bG} P_t\left(|f - f(g)|^p\right)(g) dg dt \leq 2^p \|f\|^p_p \int_{\ve}^\infty \frac{dt}{t^{\frac{s p}2 +1}}=\frac{2^{p+1}}{sp}\ve^{-\frac{s p}2}\|f\|^p_p.
\end{equation}
If we assume that
$$L \overset{def}{=}  \underset{t \to 0^+}{\limsup}\ t^{-\frac{p}{2}}\  \int_{\bG} P_t\left(|f - f(g)|^p\right)(g) dg <\infty,$$  
then we have the existence of $\ve_0>0$ such that 
\[
\underset{\tau\in (0,\ve_0)}{\sup}\ \tau^{-\frac{p}{2}}\ \int_{\bG} P_\tau\left(|f - f(g)|^p\right)(g) dg \leq L+1.
\]
This yields 
\begin{align*}
& \int_0^{\ve_0}  \frac{1}{t^{\frac{s p}2 +1}} \int_{\bG} P_t\left(|f - f(g)|^p\right)(g) dg dt \\
&\leq \left(\underset{\tau\in (0,\ve_0)}{\sup}\ \tau^{-\frac{p}{2}}\ \int_{\bG} P_\tau\left(|f - f(g)|^p\right)(g) dg \right) \int_0^{\ve_0} \frac{dt}{t^{\frac{p}{2}(s-1) +1}}\leq (L+1) \frac{2\ve_0^{\frac{p}{2}(1-s)}}{p(1 -s)}.
\end{align*}
Combining the previous estimate with \eqref{Ndopo1} (with $\ve=\ve_0$) we deduce
$$
\mathscr N_{s,p}(f)^p=\int_0^{\infty}  \frac{1}{t^{\frac{s p}2 +1}} \int_{\bG} P_t\left(|f - f(g)|^p\right)(g) dg dt \leq (L+1) \frac{2\ve_0^{\frac{p}{2}(1-s)}}{p(1 -s)} +\frac{2^{p+1}}{sp}\ve_0^{-\frac{s p}2}\|f\|^p_p
$$
which shows the validity of \eqref{inbesov}.
On the other hand, if $f\in W^{1,p}(\bG)$ we can exploit the uniform bound in \eqref{suppose} and, by arguing in the same way as before (using \eqref{Ndopo1} with $\ve=1$), obtain
\begin{equation}\label{wunoinbesov}
\mathscr N_{s,p}(f)^p \leq \frac{2C_p}{p(1 -s)}\|\nh f\|^p_p +\frac{2^{p+1}}{sp}\|f\|^p_p\qquad \forall\, f\in W^{1,p}(\bG).
\end{equation}
Finally, if $0<\sigma \leq s$ we can use $t^{-\frac{\sigma p}2 -1}\leq t^{-\frac{s p}2 -1}$ for $t\in (0,1)$ and again \eqref{Ndopo1} with $\ve=1$ in order to infer
\begin{equation}\label{besovinbesov}
\mathscr N_{\sigma,p}(f)^p \leq \mathscr N_{s,p}(f)^p +\frac{2^{p+1}}{\sigma p}\|f\|^p_p\qquad \forall\, f\in \Bps.
\end{equation}
The estimates \eqref{wunoinbesov} and \eqref{besovinbesov} complete the proof of \eqref{chain}.
\end{proof}

The next density result will be used in the proof of Theorem \ref{T:MS}.

\begin{lemma}\label{L:dens}
For every $0<s<1$ and $1\le p<\infty$, we have 
$$\overline{C^\infty_0}^{\Bps}=\Bps.$$
\end{lemma}
\begin{proof}
Following a standard pattern one can first show that
\begin{equation}\label{firstdens}
\overline{C^\infty\cap \Bps}^{\Bps}=\Bps.
\end{equation}
Once \eqref{firstdens} is accomplished, one can conclude the desired density result via multiplication with a sequence of smooth cut-off functions approximating 1 in a pointwise sense. For details we refer the reader to \cite[Prop. 3.2]{BGT}, but we mention that, in the present situation, the proof of \eqref{firstdens} is easier  with respect to that in \cite[Prop. 3.2]{BGT}. Here, given any $f\in \Bps$, in order to show that $\mathscr N_{s,p}(f-f_\ve)\to 0$ as $\ve\to 0^+$ we can use the group mollifiers $f_\ve = \rho_\ve \star f$ (see the proof of Lemma \ref{L:inf}, and in particular \eqref{unifeps}).

\end{proof}


\section{An integral decoupling property of the heat semigroup}\label{S:new}

In this section we establish some basic properties of the heat semigroup $P_t$ that will provide the backbone of Theorems \ref{T:mainp} and \ref{T:p1}. 
Let $1\le m < N\in \mathbb N$ and denote points in $\R^N$ by $(z,\sigma)$, where $z\in \Rm$ and $\sigma \in \R^{N-m}$. Consider two decoupled parabolic operators:  $L_1 - \p_t$ in $\Rm\times(0,\infty)$, and $L_2 - \p_t$ in $\R^{N-m}\times (0,\infty)$, and assume that their associated semigroups are stochastically complete.  Denote by $p_m(z,t)$ and $p_{N-m}(\sigma,t)$ their respective fundamental solutions.  It is a classical fact that the fundamental solution of the parabolic operator $L_1 + L_2 - \p_t$ in $\RN\times (0,\infty)$ is given by
\begin{equation}\label{pd}
p_N((z,\sigma),t) = p_m(z,t) \times p_{N-m}(\sigma,t).
\end{equation}
The remarkable pointwise decoupling formula \eqref{pd}, combined with the stochastic completeness of $p_{N-m}(\sigma,t)$,
 immediately implies the following integral decoupling property 
\begin{equation}\label{se}
\int_{\R^{N-m}} p_N((z,\sigma),t) d\sigma = p_m(z,t).
\end{equation}

Let us now move to sub-Riemannian geometry and consider the ``simplest" situation of a Carnot  group $\bG$ with step $r=2$. In such framework there exists the following famous formula for the heat kernel with pole at $(e,0)\in \bG\times\R$ that goes back to Gaveau-Hulanicki-Cygan, see e.g. \cite[Theor. 4.6]{GTpotan}, 
\begin{align}\label{ournucleo}
& p(g,e,t) = \frac{2^{m_2}}{(4\pi t)^{\frac Q2}}  \int_{\R^{m_2}} e^{-\frac{i}{t} \langle \sigma,\la\rangle} \left(\det j(\sqrt{- J(\la)^2})\right)^{1/2} 
\\
& \times \exp\bigg\{-\frac{1}{4t}\langle j(\sqrt{- J(\la)^2}) \cosh \sqrt{- J(\la)^2} z,z\rangle \bigg\}
d\la,\notag
\end{align}
where according with our agreement in Section \ref{S:prelim} we have identified $g\in \bG$ with its logarithmic coordinates $(z,\sigma)$, where $z\in \Rm$, and $\sigma, \la\in \R^{m_2}$. Given a $m\times m$ matrix $C$ with real coefficients, we have indicated by $j(C)$ the matrix identified by the power series of the function $j(\tau)=\frac{\tau}{\sinh(\tau)}$. Finally, we have denoted by 
$J: \bg_2 \to \operatorname{End}(\bg_1)$ the Kaplan mapping defined by 
$$\langle J(\sigma)z,\zeta\rangle = \langle [z,\zeta],\sigma\rangle = - \langle J(\sigma)\zeta,z\rangle.$$
Clearly, $J(\sigma)^\star = - J(\sigma)$, and one has $- J(\sigma)^2\ge 0$.

Formula \eqref{ournucleo} underscores how badly the sub-Riemannian heat kernel fails to have a pointwise decoupling property such as \eqref{pd} and, in general, to have any symmetry. In particular, it is not even spherically symmetric with respect to the horizontal layer $\bg_1$ of the Lie algebra. In contrast, we have the following result. 

\begin{theorem}[Integral decoupling property]\label{T:int}
Let $\bG$ be any Carnot group. For any $z\in \R^m$, $t>0$, and $(z',\sigma')\in \bG$, we have
$$
\int_{\R^{N-m}} p((z,\sigma),(z',\sigma'),t)d\sigma = (4\pi t)^{-\frac{m}{2}}e^{-\frac{||z-z'||^2}{4t}}.
$$
\end{theorem}

\begin{proof}
We observe that in view of (i) in Proposition \ref{P:prop} one has $p((z,\sigma),(z,\sigma'),t)=p((z',\sigma')^{-1}\circ (z,\sigma),e,t)$. On the other hand, from the Baker-Campbell-Hausdorff formula \eqref{BCH} we know that (see \cite[formula (1.22)]{FS} or also p. 19 in \cite{CG}) the differential of left-translation is a lower-triangular matrix with $1$'s on the main diagonal. Therefore, if we write $(z-z',\sigma'') = (z',\sigma')^{-1}\circ (z,\sigma)$, then also the Jacobian of the change of variable $\sigma \to \sigma''$ in $\R^{N-m}$ 
has Jacobian determinant one. It follows that
 \begin{equation*}
\int_{\R^{N-m}} p((z,\sigma),(z',\sigma'),t)d\sigma = \int_{\R^{N-m}} p((z-z',\sigma''),e,t)d\sigma'' = \int_{\R^{N-m}} p((z'-z,\sigma''),e,t)d\sigma''.
\end{equation*}
If we thus define
\begin{equation}\label{defkernelh}
h(z,z',t)=\int_{\R^{N-m}} p((z-z',\sigma),e,t)d\sigma,
\end{equation}
it is obvious that one has for every $z, z'\in\R^m$ and $t>0$
$$
h(z,z',t)=h(z',z,t)\qquad h(z,z',t)=h(z-z',0,t).
$$
Moreover, in view of the Gaussian estimates \eqref{gauss0}-\eqref{gauss1}-\eqref{gauss2} it is clear that $h$ is a (smooth) positive function. From property (iii) of Proposition \ref{P:prop} we have
\begin{equation}\label{fauno}
\int_{\R^m} h(z,z',t)dz'=\int_{\R^m}\int_{\R^{N-m}} p((z'-z,\sigma),e,t)d\sigma dz'=\int_\bG p(g,e,t) dg=1.
\end{equation}
Since the very definition of non-isotropic gauge yields $|(z-z',\sigma)|^2\geq ||z-z'||^2$, from \eqref{gauss0} we infer
$$
0< h(z,z',t) \leq C t^{-\frac{Q}{2}} e^{-\frac{\beta }{2} \frac{||z-z'||^2}{t} }\int_{\R^{N-m}} e^{-\frac{\beta}{2} \frac{|(z-z',\sigma)|^2}{t}}d\sigma.
$$
Using the validity of $\left(\sum_{s=2}^r ||\sigma_s||^{2r!/s}\right)^{2/2r!} \ge c_1 \sum_{s=2}^r ||\sigma_s||^{2/s}$ for some positive $c_1$, we also have
\begin{align*}
& \int_{\R^{N-m}} e^{-\frac{\beta}{2} \frac{|(z-z',\sigma)|^2}{t}}d\sigma \le \int_{\R^{N-m}} e^{-\frac{\beta c_1}{2}  \sum_{s=2}^r \frac{||\sigma_s||^{2/s}}{t}} d\sigma
\\
& = \prod_{s=2}^r \int_{\R^{m_s}} e^{-\frac{\beta c_1}{2} \frac{||\sigma_s||^{2/s}}{t}} d\sigma_s = \prod_{s=2}^r t^{\frac{s m_s}{2}} \int_{\R^{m_s}} e^{-\frac{\beta c_1}{2} ||\xi_s||^{2/s}} d\xi_s = A\ t^{\frac{Q-m}2}, 
\end{align*}
where in the second to the last equality we have made the change of variable $\sigma_s = t^{\frac s2} \xi_s$ in the integral over $\R^{m_s}$, and we have then let $A = \prod_{s=2}^r  \int_{\R^{m_s}} e^{-\frac{\beta c_1}{2} ||\xi_s||^{2/s}} d\xi_s > 0$. Substituting in the above estimate, we obtain for every $\delta>0$ and $z\in \Rm$
\begin{align*}
& \int_{\{z'\in\R^m\,:\, ||z'-z||\geq \delta\}} h(z,z',t)dz' \le C A\ t^{-\frac{m}{2}} \int_{\{z'\in\R^m\,:\, ||z'-z||\geq \delta\}} e^{-\frac{\beta }{2} \frac{||z-z'||^2}{t}} dz'
\\
& = C A \int_{\{z'\in\R^m\,:\, ||z'||\geq \frac{\delta}{t}\}} e^{-\frac{\beta }{2} ||z'||^2} dz'.
\end{align*}
From this estimate we conclude that for every $\delta>0$ and $z\in\Rm$
\begin{equation}\label{fadelta}
\underset{t \to 0^+}{\lim} \int_{\{z'\in\R^m\,:\, ||z'-z||\geq \delta\}} h(z,z',t)dz'=0.
\end{equation}
Finally, we claim that
\begin{equation}\label{facaldo}
\Delta_z h(z,z',t)-\partial_t h(z,z',t)=0\qquad\mbox{ for all }z,z'\in\R^m \mbox{ and }t>0.
\end{equation}
On the one hand, \eqref{gauss1}-\eqref{gauss2} allow to differentiate under the integral sign, obtaining
\begin{equation}\label{derh}
\begin{cases}
\partial_t h(z,z',t) =\int_{\R^{N-m}} \partial_t p((z-z',\sigma),e,t)d\sigma, 
\\
\partial_{z_i} h(z,z',t)=\int_{\R^{N-m}} \partial_{z_i} p((z-z',\sigma),e,t)d\sigma.
\end{cases}
\end{equation}
On the other hand, in the notation of Section \ref{S:prelim} we obtain from \eqref{Xi}
\[
\partial_{z_i} = X_i - \sum_{j=2}^{r}\sum_{s=1}^{m_j} b^s_{j,i}(z,...,\sigma_{j-1})
\partial_{\sigma_{j,s}}.
\]
This gives
\begin{align*}
\partial_{z_i} h(z,z',t) & = \int_{\R^{N-m}} X_i p((z-z',\sigma),e,t)d\sigma
\\
& - \sum_{j=2}^{r}\sum_{s=1}^{m_j} \int_{\R^{N-m}} b^s_{j,i}(z,...,\sigma_{j-1})
\partial_{\sigma_{j,s}} p((z-z',\sigma),e,t) d\sigma.
\end{align*}
However, recalling that $b^s_{j,i}$ are homogeneous polynomials and the derivatives of $p$ have exponential decay, we now have for $j=2,...,r$ and $s=1,...,m_j$
\begin{align*}
& \int_{\R^{N-m}} b^s_{j,i}(z,...,\sigma_{j-1})
\partial_{\sigma_{j,s}} p((z-z',\sigma),e,t) d\sigma
\\
&=\int_{\R^{N-m-1}} b^s_{j,i}(z,...,\sigma_{j-1}) \left(\int_\R \partial_{\sigma_{j,s}} p((z-z',\sigma),e,t) d\sigma_{j,s}\right) d\hat{\sigma} \\
&=\int_{\R^{N-m-1}} b^s_{j,i}(z,...,\sigma_{j-1}) \bigg\{p((z-z',\sigma),e,t)\bigg\}^{\sigma_{j,s}=\infty}_{\sigma_{j,s}=-\infty} d\hat{\sigma}=0,
\end{align*}
where the last equality is again a consequence of \eqref{gauss0}. This implies the remarkable identity
$$
\partial_{z_i} h(z,z',t) = \int_{\R^{N-m}} X_i p((z-z',\sigma),e,t)d\sigma.
$$
Arguing in the same way for $\partial_{z_i z_i} h$ and summing in $i\in\{1,\ldots,m\}$, we obtain
\begin{equation}\label{der2h}
\Delta_z h(z,z',t)=\int_{\R^{N-m}} \mathscr L  p((z-z',\sigma),e,t)d\sigma,
\end{equation}
where $\mathscr L$ is as in \eqref{L}. 
Since $\left(\mathscr L -\partial_t\right) p=0$, the combination of \eqref{derh} and \eqref{der2h} shows the validity of the claim \eqref{facaldo}. Let now $\vf \in C(\R^m)\cap L^{\infty}(\R^m)$. From the crucial properties \eqref{fauno},\eqref{fadelta} and \eqref{facaldo} (and the ubiquitous Gaussian estimates \eqref{gauss0}-\eqref{gauss1}-\eqref{gauss2}) we infer in a standard fashion that the bounded function
$$
u(z,t)\overset{def}{=}\int_{\R^m} h(z,z',t)\vf(z')dz'
$$
is a classical solution to the Cauchy problem in $\R^m\times(0,\infty)$
\begin{equation}\label{cp}
\begin{cases}
\Delta_z u -  \partial_t u = 0,
\\
u(z,0) = \vf(z).
\end{cases}
\end{equation}
From the uniqueness of the bounded solutions to \eqref{cp} we infer that
$$
\int_{\R^m} h(z,z',t)\vf(z')dz'= \int_{\R^m}(4\pi t)^{-\frac{m}{2}}e^{-\frac{||z-z'||^2}{4t}}\vf(z')dz'.
$$
Since this identity holds for arbitrary $z\in\R^m$, $t>0$, and $\vf \in C(\R^m)\cap L^{\infty}(\R^m)$, we conclude that for every $z, z'\in\R^m$ and $t>0$
$$
h(z,z',t)=(4\pi t)^{-\frac{m}{2}}e^{-\frac{||z-z'||^2}{4t}}.
$$
Recalling \eqref{defkernelh}, we have thus completed the proof of the theorem. 

\end{proof}

\subsubsection{{\bf Nihil sub sole novum (Ecclesiaste), or also ``Chi cerca trova, e chi ricerca... ritrova" (E. De Giorgi)}}\label{SS:fulvio}

We thank Fulvio Ricci for kindly bringing to our attention that a different, more abstract proof of the integral decoupling formula in Theorem \ref{T:int} can be extracted from the following result of D. M\"uller in \cite[Proposition 1.1]{Mu}.  Let $C_\infty(\R^+)$ be the space of continuous functions on $\R^+$ which vanish at infinity.

\begin{proposition}[D. M\"uller]
Let $\bG$ be a homogeneous Lie group and let $L$ be a left-invariant, positive Rockland differential operator on $\bG$.
 Let $\pi$ be an irreducible unitary representation of $\bG$. If $m\in C_\infty(\R^+)$, then
\[
\pi(m(L)) = m(d\pi(L)).
\]
\end{proposition}


\subsection{A sub-Riemannian Ledoux-Huisken lemma}\label{SS:LH}
We next use Theorem \ref{T:int} to establish another remarkable property of the heat semigroup $P_t$. To motivate it, we consider the standard heat kernel in $\Rm$, $p(z,t) = (4\pi t)^{-\frac m2} e^{-\frac{|z|^2}{4t}}$, and let $\mathbb S^{m-1} = \{\nu\in \Rm\mid ||\nu|| = 1\}$. It is an elementary (and beautiful) exercise to show that for any $\nu\in \mathbb S^{m-1}$, $t>0$, and $1\le p<\infty$ one has 
\begin{equation}\label{led}
\frac{1}{t^{\frac p2}} \int_{\Rm} p(z,t) |\langle \nu,z\rangle|^p dz = \frac{2 \G(p)}{\G(p/2)}.
\end{equation}
The reader should note the appearance in the right-hand side of \eqref{led} of the same dimensionless constant in the right-hand side of \eqref{thesispPtk}. As far as we are aware of, the identity \eqref{led} was first used by Ledoux in the case $p=1$ in his approach to the isoperimetric inequality based on the heat semigroup, see \cite{Led} and also \cite{MPPP}. Another elementary (and equally beautiful) property of the standard heat kernel in $\Rm$ is expressed by the following identity which represents a special case of a deep monotonicity formula discovered by Huisken, see \cite[Theor. 3.1]{Hui}, 
\begin{equation}\label{hui}
\sqrt{4\pi t} \int_{\{z\in\R^m\mid\left\langle z,\nu\right\rangle =0\}} p(z,t)\ dz = 1,
\end{equation} 
where again $\nu\in \mathbb S^{m-1}$ is an arbitrary direction and $t>0$. We note that the hyperplane $\{z\in\R^m\mid\left\langle z,\nu\right\rangle =0\}$ is a global minimal surface in $\Rm$. In both \eqref{led} and \eqref{hui} the symmetries of the heat kernel seem to play an essential role. Quite surprisingly, the next result shows that, in fact, \eqref{led} and \eqref{hui} have a universal character, in the sense that they hold unchanged when the Gauss-Weierstrass kernel is replaced by the very complicated and non-explicit heat kernel $p(g,e,t)$ in any Carnot group! As agreed in the opening of the section, we will identify a point $g\in \bG$ with $(z,\sigma)$, where $z\in \Rm$ and $\sigma\in \R^{N-m}$.

\begin{lemma}[sub-Riemannian Ledoux-Huisken lemma]\label{L:id}
Let $\bG$ be a Carnot group, and let $1\le p<\infty$. For any $\nu\in \mathbb S^{m-1}$ and $t>0$ we have
\begin{equation}\label{Inucostante}
\frac{1}{t^{\frac p2}} \int_{\bG} p(g,e,t) |\langle \nu,z\rangle|^p dg = \frac{2 \G(p)}{\G(p/2)},
\end{equation}
and
\begin{equation}\label{punoint}
\sqrt{4\pi t} \int_{\{z\in\R^m\mid\left\langle z,\nu\right\rangle =0\}\times \R^{N-m}} p((z,\sigma),e,t)\ dzd\sigma = 1.
\end{equation}
\end{lemma}

\begin{proof}
The proofs of \eqref{Inucostante} and \eqref{punoint} follow immediately by an application of Theorem \ref{T:int} with $(z',\sigma')=e$. To see this, if we let $g = (z,\sigma)$, then by Tonelli's theorem and Theorem \ref{T:int} we have
\begin{align*}
& \frac{1}{t^{\frac p2}} \int_{\bG} p((z,\sigma),e,t) |\langle \nu,z\rangle|^p dzd\sigma = \frac{1}{t^{\frac p2}} \int_{\R^m} |\langle \nu,z\rangle|^p \left(\int_{\R^{N-m}} p((z,\sigma),e,t) d\sigma\right)dz\\
&=\frac{1}{t^{\frac p2}} \int_{\R^m}|\langle \nu,z\rangle|^p (4\pi t)^{-\frac m2} e^{-\frac{||z||^2}{4}} dz = \frac{2 \G(p)}{\G(p/2)},
\end{align*}
where in the last equality we have used \eqref{led}. In a similar fashion and using \eqref{hui}, we have 
\begin{align*}
&\int_{\{z\in\R^m\mid\left\langle z,\nu\right\rangle =0\}\times \R^{N-m}} p((z,\sigma),e,t)dzd\sigma\\
&=\int_{\{z\in\R^m\mid\left\langle z,\nu\right\rangle =0\}} (4\pi t)^{-\frac m2}e^{-\frac{||z||^2}{4t}} dz =\frac{1}{\sqrt{4\pi t}}.
\end{align*}\end{proof}
In connection with \eqref{punoint} it is worth remarking here that in the Heisenberg group $\mathbb H^1$ the vertical planes $\{z\in\R^2\mid\left\langle z,\nu\right\rangle =0\}\times \R$ are the only stable $H$-minimal entire graphs with empty characteristic locus, see \cite[Theor. 1.8]{DGNP}.


\section{Proof of Theorems \ref{T:mainp} and \ref{T:p1}}\label{S:proofs}

This section is devoted to proving Theorems \ref{T:mainp} and \ref{T:p1}. In order to establish Theorem \ref{T:mainp} we are going to show in the next two lemmas a ``limsup" and a ``liminf'' inequality in the spirit of the arguments in \cite[Sec. 2]{B}, where the case of a spherically symmetric approximate identity $\{\rho_\ve\}$ is treated. One should also see the generalisation to Carnot groups in \cite[Sec. 3]{Bar}. As the reader will see, in our situation the accomplishment of this task is possible thanks to the remarkable identity \eqref{Inucostante} in Lemma \ref{L:id} which ultimately hinges on Theorem \ref{T:int}.

\begin{lemma}\label{L:sup}
Let $f\in W^{1,p}(\bG)$ with $1\leq p<\infty$. Then
\begin{equation}\label{charlimsup}
\underset{t \to 0^+}{\limsup}\  t^{-\frac{p}{2}}\ \int_{\bG} P_t\left(|f - f(g)|^p\right)(g) dg\leq \frac{2 \G(p)}{\G(p/2)} \|\nabla_H f\|^p_p.
\end{equation}
\end{lemma}
\begin{proof}
We begin by proving  \eqref{charlimsup} for functions in $C^\infty_0(\bG)$.
For a fixed  $f\in C_0^\infty(\bG)$ we denote by $K=\overline \Om$, where $\Om$ is a bounded open set such that ${\rm{supp}}f\subset \Om$. Since $\left| (g')^{-1}\circ g\right|\ge \gamma>0$ for every $g\in \operatorname{supp} f$ and $g'\in \bG\setminus K$, the estimate \eqref{gauss0} in Proposition \ref{P:gaussian} easily implies
\begin{align*}
& \underset{t \to 0^+}{\lim}\  t^{-\frac{p}{2}}\ \int_{\bG\smallsetminus K}\int_{{\rm{supp}}f} p(g,g',t)|f(g') - f(g)|^p dg'dg
\\
&=\underset{t \to 0^+}{\lim}\  t^{-\frac{p}{2}}\ \int_{{\rm{supp}}f} \int_{\bG\smallsetminus K} p(g,g',t)|f(g') - f(g)|^p dg'dg=0.
\end{align*}
Since this gives
$$
\underset{t \to 0^+}{\limsup}\  t^{-\frac{p}{2}}\ \int_{\bG} P_t\left(|f - f(g)|^p\right)(g) dg=\underset{t \to 0^+}{\limsup}\  t^{-\frac{p}{2}}\ \int_{K}\int_{K} p(g,g',t)|f(g') - f(g)|^p dg'dg,
$$
we are thus left with proving that
\begin{equation}\label{claimlimsup}
\underset{t \to 0^+}{\limsup}\ t^{-\frac{p}{2}} \int_{K}\int_{K} p(g,g',t)|f(g') - f(g)|^p dg'dg \leq \frac{2 \G(p)}{\G(p/2)} \|\nabla_H f\|^p_p.
\end{equation}
From the uniform continuity and boundedness of $\nh f$, it is clear from Proposition \ref{P:taylor} that for every $g, g'\in K$ one has 
$$f(g') - f(g) = \langle\nh f(g),z'-z\rangle +   \omega(g,g')$$
where 
\begin{equation}\label{opiccolo}
|\omega(g,g')| = o(d(g,g'))\mbox{   as  }d(g,g')\to 0\,\,\mbox{ and }\,\, \frac{|\omega(g,g')|}{d(g,g')} \mbox{  is bounded uniformly in }K.
\end{equation}
Hence, for any $\theta>0$ there exists $c_\theta>0$ such that
$$
|f(g') - f(g)|^p\leq (1+\theta) \left| \left\langle \nabla_{H} f(g), z'-z \right\rangle \right|^p + c_\theta\, |\omega(g,g')|^p.
$$
This gives
\begin{align}\label{daqui}
&t^{-\frac{p}{2}}\int_{K}\int_{K} p(g,g',t)|f(g') - f(g)|^p dg'dg\\
&\leq (1+\theta) t^{-\frac{p}{2}}\int_{K}\int_{K} p(g,g',t)\left| \left\langle \nabla_{H} f(g), z'-z \right\rangle \right|^p dg'dg +\notag\\
&+ c_\theta\, t^{-\frac{p}{2}}\int_{K}\int_{K} p(g,g',t)  |\omega(g,g')|^pdg'dg = I(t) + II(t).\notag
\end{align}
Keeping in mind that by (i) in Proposition \ref{P:prop} we have $p(g,g',t) = p(g',g,t) = p(g^{-1}\circ g',e,t)$, if we now use the change of variable $g'' = g^{-1}\circ g'$, we obtain 
\begin{align*}
I(t) & =  (1+\theta)t^{-\frac{p}{2}}\int_{K}\int_{K} p(g,g',t)|\langle \nabla_{H} f(g), z'-z\rangle|^p dg' dg 
\\
& = (1+\theta)t^{-\frac{p}{2}}\int_{\bG}\int_{\bG} p(g^{-1}\circ g',e,t)|\langle \nabla_{H} f(g), z'-z\rangle|^p dg' dg 
\\
& =  (1+\theta)t^{-\frac{p}{2}}\int_{\{g\in \bG\,:\, |\nabla_{H} f(g)|\not=0\}}\int_{\bG} p(g'',e,t)|\langle \nabla_{H} f(g), z''\rangle|^p dg'' dg\\
&=  (1+\theta)\frac{2 \G(p)}{\G(p/2)} \int_{\bG}|\nabla_{H} f(g)|^p dg,
\end{align*}
where in the last equality we have used \eqref{Inucostante} in Lemma \ref{L:id}. On the other hand, if we now fix a sufficiently small $\delta>0$, we have from \eqref{opiccolo}
\begin{align*}
&II(t)  \le  c_\theta\, t^{-\frac{p}{2}}\int_{K}\int_{B(g,\delta)} p(g,g',t) |o(d(g,g'))|^p dg'dg +\\
&+ \bar{C} t^{-\frac{p}{2}}\int_{K}\int_{K\setminus B(g,\delta)} p(g,g',t) \left| (g')^{-1}\circ g\right|^p dg'dg\\
&\leq c_\theta\,|\omega(t)|\int_{K}\int_{\bG} p(e,g'',1)|g''|^p dg'' dg + \bar{C} \int_{K}\int_{\bG\smallsetminus B\left(e,\frac{\delta}{\sqrt{t}}\right)} p(e,g'',1)|g''|^p dg'' dg \\
& \leq c_\theta\,|\omega(t)| |K| C\int_{\bG} e^{-\beta |g''|^2}|g''|^p dg'' + \bar{C} |K| C e^{-\frac{\beta\delta^2}{2t}} \int_{\bG} e^{-\frac{1}{2}\beta |g''|^2}|g''|^p dg'',
\end{align*}
where $|\omega(t)|\to 0$ as $t\to 0^+$, $\bar{C}$ is a suitable positive constant, and in the relevant integrals we have used Propositions \ref{P:prop} and \ref{P:gaussian}. The analysis of $I(t)$ and $II(t)$ implies  
\[
I(t) =  (1+\theta) \frac{2 \G(p)}{\G(p/2)} \int_{\bG}|\nabla_{H} f(g)|^p dg,\ \ \ \ \ \ \underset{t \to 0^+}{\lim}\ II(t) =0.
\]
Substituting these relations in \eqref{daqui} we conclude that 
$$\underset{t \to 0^+}{\limsup}\ t^{-\frac{p}{2}} \int_{K}\int_{K} p(g,g',t)|f(g') - f(g)|^p dg'dg \leq (1+\theta) \frac{2 \G(p)}{\G(p/2)} \int_{\bG}|\nabla_{H} f(g)|^p dg.$$
The arbitrariness of $\theta>0$ leads to the validity of \eqref{claimlimsup} (and therefore of \eqref{charlimsup}) for $f\in C^\infty_0(\bG)$.  In order to complete the proof we are going to use the density property \eqref{dense}. Given $f\in W^{1,p}(\bG)$, there exists $f_k\in C_0^\infty$ such that $||f-f_k||_{W^{1,p}(\bG)} \to 0$. This gives
\begin{align*}
& \underset{t \to 0^+}{\limsup}\   t^{-\frac{p}{2}} \int_{\bG} P_t\left(|f - f(g)|^p\right)(g) dg
 \le \underset{t \to 0^+}{\limsup}\   t^{-\frac{p}{2}} \int_{\bG} P_t\left(|(f-f_k) - (f(g)-f_k(g)|^p\right)(g) dg
 \\
 & + \underset{t \to 0^+}{\limsup}\   t^{-\frac{p}{2}} \int_{\bG} P_t\left(|f_k - f_k(g)|^p\right)(g) dg
 \le C_p ||\nabla_H f_k - \nabla_H f||^p_{L^p(\bG)} + \frac{2 \G(p)}{\G(p/2)} ||\nabla_H f_k||^p_{L^p(\bG)},
\end{align*}
where in the last inequality we have applied \eqref{suppose} in Lemma \ref{L:diffquot} and \eqref{charlimsup}. Letting $k\to \infty$ we conclude that \eqref{charlimsup} does hold also for $f\in W^{1,p}(\bG)$. This completes the proof.

\end{proof}

The next lemma provides the converse implication of Lemma \ref{L:sup}, and it requires $p>1$.

\begin{lemma}\label{L:inf}
Let $p>1$ and $f\in L^p(\bG)$. Suppose that
$$\underset{t \to 0^+}{\liminf}\ t^{-\frac{p}{2}}\ \int_{\bG} P_t\left(|f - f(g)|^p\right)(g) dg <\infty.$$
Then $f\in W^{1,p}(\bG)$, and one has
\begin{equation}\label{charliminf}
\frac{2 \G(p)}{\G(p/2)} \|\nabla_H f\|^p_p\leq \underset{t \to 0^+}{\liminf}\  t^{-\frac{p}{2}}\ \int_{\bG} P_t\left(|f - f(g)|^p\right)(g) dg.
\end{equation}
\end{lemma}

\begin{proof}
Fix a nonnegative function $\rho\in C_0^\infty(\bG)$ such that $\int_{\bG}\rho(g) dg =1$, and denote by
$$
\rho_\ve (g)=\ve^{-Q} \rho(\delta_{\ve^{-1}}(g)),
$$
the approximate identity associated with $\rho$. Let $1\le p<\infty$ (the reader should note that we are not excluding the case $p=1$ at this moment). For any $f\in L^p(\bG)$ we also let
$$
f_\ve (g)= \rho_\ve \star f(g) = \int_{\bG} \rho_\ve(g') f( (g')^{-1}\circ g) dg' = \int_{\bG} \rho_\ve(g\circ (g')^{-1}) f(g') dg'.
$$
It is well-known, see e.g. \cite[Prop. 1.20]{FS} and \cite[Theor. 2.3]{Ri}, that $f_\ve \in L^p(\bG)\cap C^{\infty}(\bG)$ and $f_\ve \to f$ in $L^p(\bG)$ as $\ve\to 0^+$. The following simple, yet critical, fact holds for every $\ve >0$ and $t>0$
\begin{equation}\label{unifeps}
{t}^{-\frac{p}{2}}\ \int_{\bG} P_{t}\left(|f_\ve - f_\ve(g)|^p\right)(g) dg \leq {t}^{-\frac{p}{2}}\ \int_{\bG} P_{t}\left(|f - f(g)|^p\right)(g) dg.
\end{equation}
To see \eqref{unifeps} note that H\"older inequality and (i) of Proposition \ref{P:prop} imply for any $t>0$
\begin{align*}
&\int_{\bG} P_t\left(|f_\ve - f_\ve(g)|^p\right)(g) dg= \int_{\bG}\int_{\bG} p(g,g',t) |f_\ve(g') - f_\ve(g)|^p dg'dg\\
&\leq \int_{\bG}\int_{\bG} p(g,g',t)\int_{\bG} \rho_\ve(g'') |f( (g'')^{-1}\circ g') - f( (g'')^{-1}\circ g)|^p dg'' dg'dg\\
&= \int_{\bG} \rho_\ve(g'')\left(\int_{\bG}\int_{\bG} p((g'')^{-1}\circ g,(g'')^{-1}\circ g',t)  |f( (g'')^{-1}\circ g') - f( (g'')^{-1}\circ g)|^p dg'dg\right)dg''\\
&= \int_{\bG} \rho_\ve(g'') \left(\int_{\bG} P_t\left(|f - f(g)|^p\right)(g) dg\right) dg'' = \int_{\bG} P_t\left(|f - f(g)|^p\right)(g) dg.
\end{align*}
We now claim that for every $1\le p < \infty$ and $\ve>0$ we have
\begin{equation}\label{evvivaevviva}
\frac{2 \G(p)}{\G(p/2)} \int_\bG |\nabla_{H} f_\ve(g)|^p dg \leq \underset{t \to 0^+}{\liminf}\ t^{-\frac{p}{2}}\ \int_{\bG} P_t\left(|f - f(g)|^p\right)(g) dg<\infty.
\end{equation}
The advantage now is that we can use smoothness and therefore argue similarly to the proof of Lemma \ref{L:sup}. We fix $\ve>0$ and let $K\subset \bG$ be a given compact set. From the continuity of $\nh f_\ve$ and Proposition \ref{P:taylor} we know that that for every $g\in K$
\[
\langle\nh f_\ve(g),z'-z\rangle = f_\ve(g') - f_\ve(g)  + \omega_\ve(g,g'),
\]
where $|\omega_\ve(g,g')| = o(d(g,g'))$ as $g'\to g$ uniformly for $g\in K$. As in Lemma \ref{L:sup}, for any $\theta>0$ there exists $c_\theta>0$ such that
\begin{equation}\label{unactheta}
|\langle \nabla_{H} f_\ve(g),z'-z\rangle|^p\leq (1+\theta)|f_\ve(g') - f_\ve(g)|^p + c_\theta |\omega_\ve(g,g')|^p.
\end{equation}
From \eqref{Inucostante} in Lemma \ref{L:id} we see that for any $0<\rho<\frac{2 \G(p)}{\G(p/2)}$ there exists $M=M(\rho)>0$ such that for any $\nu\in \mathbb S^{m-1}$ one has
\begin{equation}\label{inpartgauss}
\frac{2 \G(p)}{\G(p/2)}-\rho \leq \int_{B(e,M)} p(g'',e,1) |\langle \nu,z''\rangle|^p dg'',
\end{equation}
where the uniformity of the constant $M$ with respect to $\nu$ is ensured by \eqref{gauss0}. Assuming that $g\in K$ is a point at which $|\nabla_{H} f_\ve(g)|\not= 0$, applying \eqref{inpartgauss} with $\nu= \frac{\nabla_{H} f_\ve(g)}{|\nabla_{H} f_\ve(g)|}$, we find
\begin{align*}
& \left(\frac{2 \G(p)}{\G(p/2)}-\rho\right) |\nabla_{H} f_\ve(g)|^p \le \int_{B(e,M)} p(g'',e,1) |\langle \nabla_{H} f_\ve(g),z''\rangle|^p dg''.
\end{align*}
Since such inequality continues to be trivially true at points $g\in K$ where $|\nabla_{H} f_\ve(g)|= 0$, we can assume that it holds for every $g\in K$. Integration in $g\in K$, together with the change of variable $\zeta\mapsto g'=g\circ \delta_{\sqrt{t}}g''$ and the estimate \eqref{unactheta}, give
\begin{align*}
& \left(\frac{2 \G(p)}{\G(p/2)}-\rho\right) \int_K |\nabla_{H} f_\ve(g)|^p dg \le \int_K \int_{B(e,M)} p(g'',e,1) |\langle \nabla_{H} f_\ve(g),z''\rangle|^p dg'' dg.
\\
&=t^{-\frac{p}{2}}\ \int_K\int_{B(g,\sqrt{t}M)} p(g,g',t) |\langle \nabla_{H} f_\ve(g),z'-z\rangle|^p dg'dg\\
&\leq (1+\theta)t^{-\frac{p}{2}}\ \int_K\int_{B(g,\sqrt{t}M)} p(g,g',t)|f_\ve(g') - f_\ve(g)|^p dg'dg 
\\
& + c_\theta t^{-\frac{p}{2}}\ \int_K\int_{B(g,\sqrt{t}M)} p(g,g',t)|\omega_\ve(g,g')|^p dg'dg
 \\
&\leq (1+\theta) t^{-\frac{p}{2}} \int_{\bG} P_t\left(|f_\ve - f_\ve(g)|^p\right)(g) dg + c_\theta t^{-\frac{p}{2}}\ \int_K\int_{B(g,\sqrt{t}M)} p(g,g',t)|\omega_\ve(g,g')|^p dg'dg
\\
&\leq (1+\theta) t^{-\frac{p}{2}} \int_{\bG} P_t\left(|f - f(g)|^p\right)(g) dg + +c_\theta t^{-\frac{p}{2}}\ \int_K\int_{B(g,\sqrt{t}M)} p(g,g',t)|\omega_\ve(g,g'))|^p dg'dg,
\end{align*}
where in the last inequality we used the key property \eqref{unifeps}. Taking the $\underset{t \to 0^+}{\liminf}$ of both sides of the latter inequality, we find
\begin{align*}
& \left(\frac{2 \G(p)}{\G(p/2)}-\rho\right) \int_K |\nabla_{H} f_\ve(g)|^p dg\ \le\ (1+\theta)\ \underset{t \to 0^+}{\liminf}\ t^{-\frac{p}{2}} \int_{\bG} P_t\left(|f - f(g)|^p\right)(g) dg
\\
&   + c_\theta\ \underset{t \to 0^+}{\limsup}\ t^{-\frac{p}{2}}\ \int_K\int_{B(g,\sqrt{t}M)} p(g,g',t)|\omega_\ve(g,g'))|^p dg'dg.
\end{align*}
By \eqref{gauss0} in Proposition \ref{P:gaussian} and the property of $\omega_\ve$, it is now easy to infer that, for any compact $K$ and any fixed $M>0$, we have
$$
\underset{t \to 0^+}{\lim}\ t^{-\frac{p}{2}}\ \int_K\int_{B(g,\sqrt{t}M)} p(g,g',t)|\omega_\ve(g,g'))|^p dg'dg = 0.
$$
We have thus proved that
$$
\left(\frac{2 \G(p)}{\G(p/2)}-\rho\right) \int_K |\nabla_{H} f_\ve(g)|^p dg\ \le\ (1+\theta)\ \underset{t \to 0^+}{\liminf}\ t^{-\frac{p}{2}} \int_{\bG} P_t\left(|f - f(g)|^p\right)(g) dg.
$$
By the arbitrariness of $K, \theta, \rho$, we conclude that \eqref{evvivaevviva} does hold. For later use in the proof of Theorem \ref{T:p1} below, we reiterate at this point that  \eqref{evvivaevviva} is valid for any $1\le p<\infty$. 

We can now complete the proof of the lemma. By the theorem of Banach-Alaoglu (this is the only place where we are using $p>1$!) we know that  (up to subsequence) for every $i=1,...,m$, $X_i f_\ve$ is weakly-convergent as $\ve\to 0^+$ to a function $g_i\in L^p(\bG)$. On the other hand, since $f_\ve\to f$ in $L^p(\bG)$, for every $i=1,...,m$, we have that $X_i f_\ve$ converges to $X_i f$ in $\mathscr D'(\bG)$. This shows $X_i f=g_i$ as $L^p$-functions, thus proving that $f\in W^{1,p}(\bG)$. Once we know this crucial information, we use the fact that, from the definition of $f_\ve$, we have $X_i(f_\ve) = \rho_\ve\star(X_i f)= (X_i f)_\ve$ for $i=1,...,m$, and therefore $\nabla_{H}(f_\ve)=\left(\nabla_{H} f\right)_\ve\to \nabla_{H} f$ in $L^p(\bG)^m$ as $\ve\to 0^+$. This allows to infer from \eqref{evvivaevviva} that
$$
\frac{2 \G(p)}{\G(p/2)}\int_\bG |\nabla_{H} f(g)|^p dg \le \underset{t \to 0^+}{\liminf}\  t^{-\frac{p}{2}}\ \int_{\bG} P_t\left(|f - f(g)|^p\right)(g) dg,
$$
which finally proves \eqref{charliminf}.

\end{proof}

We are then in a position to provide the 
\begin{proof}[Proof of Theorem \ref{T:mainp}] 
It follows immediately by combining Lemmas \ref{L:sup} and \ref{L:inf}. 

\end{proof}

We next turn the attention to the 

\begin{proof}[Proof of Theorem \ref{T:p1}] 
We begin with proving that for $f\in L^1(\bG)$, we have
\begin{equation}\label{infuno}
\underset{t \to 0^+}{\liminf}\ \frac{1}{\sqrt t}\ \int_{\bG} P_t\left(|f - f(g)|\right)(g) dg<\infty\ \Longrightarrow\ f\in BV(\bG).
\end{equation}
Keeping in mind that \eqref{evvivaevviva} in the proof of Lemma \ref{L:inf} does hold also when $p=1$, we obtain from it for all $\ve>0$
$$
\frac{2}{\sqrt{\pi}} \int_\bG |\nabla_{H} f_\ve(g)| dg \leq \underset{t \to 0^+}{\liminf}\ \frac{1}{\sqrt t}\ \int_{\bG} P_t\left(|f - f(g)|\right)(g) dg<\infty,
$$
where as before $f_\ve=\rho_\ve\star f \in C^\infty(\bG)\cap L^1(\bG)$ is the family of group mollifiers. This shows in particular that $f_\ve \in W^{1,1}(\bG)$. Recalling definition \eqref{var}, if we now fix $\zeta = (\zeta_1,...,\zeta_m)\in \mathscr F$, we clearly have for all $\ve >0$
\begin{align}\label{evvivaevviva1}
&\int_{\bG} f_\ve (g) \sum_{i=1}^m X_i \zeta_i (g)  dg = - \int_{\bG} \left\langle \nh f_\ve (g), \zeta(g) \right\rangle  dg \\
&\leq \int_\bG |\nabla_{H} f_\ve(g)| dg \leq \frac{\sqrt{\pi}}{2}\underset{t \to 0^+}{\liminf}\ \frac{1}{\sqrt t}\ \int_{\bG} P_t\left(|f - f(g)|\right)(g) dg.\notag
\end{align}
Using the fact that $f_\ve \to f$ in $L^1(\bG)$ as $\ve \to 0^+$, together with \eqref{evvivaevviva1}, we obtain
\begin{align*}
& \int_{\bG} f(g) \sum_{i=1}^m X_i \zeta_i (g)  dg = \underset{\ve \to 0^+}{\lim} \int_{\bG} f_\ve (g) \sum_{i=1}^m X_i \zeta_i (g) dg
\\
&  \leq \frac{\sqrt{\pi}}{2}\underset{t \to 0^+}{\liminf}\ \frac{1}{\sqrt t}\ \int_{\bG} P_t\left(|f - f(g)|\right)(g) dg.
\end{align*}
The arbitrariness of $\zeta \in \mathscr F$ and \eqref{var} yield the validity of \eqref{infuno}, 
and we conclude that
\begin{equation}\label{infunosharp}
\frac{2}{\sqrt{\pi}} \operatorname{Var}_\bG(f) \leq \underset{t \to 0^+}{\liminf}\ \frac{1}{\sqrt t}\ \int_{\bG} P_t\left(|f - f(g)|\right)(g) dg\ .
\end{equation}
The opposite implication, 
\begin{equation}\label{supuno}
f\in BV(\bG)\ \Longrightarrow\ \underset{t \to 0^+}{\limsup}\ \frac{1}{\sqrt t}\ \int_{\bG} P_t\left(|f - f(g)|\right)(g) dg<\infty,
\end{equation}
is a trivial consequence of Proposition \ref{P:minchietto}. The two implications \eqref{infuno} and \eqref{supuno} prove \eqref{1uno}. If, in addition, $f\in W^{1,1}(\bG)$, then keeping in mind that $\operatorname{Var}_\bG(f) = \int_{\bG} |\nh f(g)| dg$, and by combining \eqref{infunosharp} with the case $p=1$ of Lemma \ref{L:sup}, we obtain
$$
\exists\underset{t \to 0^+}{\lim}\ \frac{1}{\sqrt t}\ \int_{\bG} P_t\left(|f - f(g)|\right)(g) dg = \frac{2}{\sqrt{\pi}} \int_\bG |\nabla_{H} f(g)| dg.
$$
This proves \eqref{2unouno}. 
We are left with the proof of \eqref{2uno}. We first apply Theorem \ref{T:bmpB}, that gives for  $f\in BV(\bG)$
\begin{align}\label{olliop}
&\underset{t \to 0^+}{\lim}\ \frac{1}{\sqrt t}\ \int_{\bG} P_t\left(|f - f(g)|\right)(g) dg =
4 \int_{\bG} \int_{T_\bG(\sigma_f(g))} p(g',e,1) dg' d \operatorname{Var}_\bG(f)(g)
\\
&= \int_{\bG} \int_{T_{\sigma_f(g)}\times \R^{N-m}} p(g',e,1) dg' d \operatorname{Var}_\bG(f)(g),
\notag
\end{align}
where in the last equality we have used \eqref{Tnu} with $\nu = \sigma_f(g)$. It is at this point that the crucial identity \eqref{punoint} in Lemma \ref{L:id} enters the stage. Using it we find that for $\operatorname{Var}_\bG(f)$-a.e. $g\in \bG$ one has
\begin{equation}\label{wow}
4 \int_{T_{\sigma_f(g)}\times \R^{N-m}} p(g',e,1) dg' = \frac{2}{\sqrt \pi}.
\end{equation}
Inserting \eqref{wow} in \eqref{olliop} we obtain \eqref{2uno}, thus completing the proof of the theorem.

\end{proof}


\section{Limiting behaviour of Besov seminorms}\label{S:seminorms}

In this section we prove Theorems \ref{T:bbmG} and \ref{T:MS}. As we have implicitly mentioned in the introduction, the gist of the proof of Theorem \ref{T:bbmG} is to connect the limit as $t\searrow 0^+$ of the heat semigroup in Theorem \ref{T:mainp} with that as $s\nearrow 1^-$ of the desingularised Besov seminorms $(1-s) \mathscr N_{s,p}(f)^p$. This connection is expressed by the following proposition (see also \cite[Section 3]{GTbbmd} for the case $p=1$).

\begin{proposition}\label{P:stars}
Let $1\leq p<\infty$ and $f\in L^p(\bG)$. One has
\begin{align}\label{chaininfsup}
&\frac{2}{p}\ \underset{t \to 0^+}{\liminf}\ t^{-\frac{p}{2}}\ \int_{\bG} P_t\left(|f - f(g)|^p\right)(g) dg \leq \underset{s\to 1^-}{\liminf}\ (1 - s)\ \mathscr N_{s,p}(f)^p  \leq \\
&\leq \underset{s\to 1^-}{\limsup}\ (1 - s)\ \mathscr N_{s,p}(f)^p   \leq \frac{2}{p}\ \underset{t \to 0^+}{\limsup}\ t^{-\frac{p}{2}}\ \int_{\bG} P_t\left(|f - f(g)|^p\right)(g) dg.\notag
\end{align}
\end{proposition}
\begin{proof}
We start by proving
\begin{equation}\label{ve2}
\underset{s\to 1^-}{\limsup}\ (1 - s)\ \mathscr N_{s,p}(f)^p   \leq \frac{2}{p}\ \underset{t \to 0^+}{\limsup}\ t^{-\frac{p}{2}}\ \int_{\bG} P_t\left(|f - f(g)|^p\right)(g) dg,
\end{equation}  
and we denote 
$$L \overset{def}{=}  \underset{t \to 0^+}{\limsup}\ t^{-\frac{p}{2}}\  \int_{\bG} P_t\left(|f - f(g)|^p\right)(g) dg.$$  
If $L=+\infty$, then \eqref{ve2} is trivially valid. Therefore, we might as well assume that $L<\infty$. By \eqref{inbesov} we already know that $f$ belongs to the fractional Sobolev space $\Bps$ for every $0<s<1$. Fix $\ve_0>0$ such that 
\[
\underset{\tau\in (0,\ve_0)}{\sup}\ \tau^{-\frac{p}{2}}\ \int_{\bG} P_\tau\left(|f - f(g)|^p\right)(g) dg <\infty
\]
and consider any $\ve\in(0,\ve_0)$. Hence we can write
\begin{align*}
\mathscr N_{s,p}(f)^p & = \int_0^\ve \frac{1}{t^{\frac{s p}2 +1}} \int_{\bG} P_t\left(|f - f(g)|^p\right)(g) dg dt +  \int_\ve^\infty \frac{1}{t^{\frac{s p}2 +1}} \int_{\bG} P_t\left(|f - f(g)|^p\right)(g) dg dt.
\end{align*}
On one hand, by \eqref{Ndopo1} we know
\[
 \int_\ve^\infty \frac{1}{t^{\frac{s p}2 +1}} \int_{\bG} P_t\left(|f - f(g)|^p\right)(g) dg dt \leq \frac{2^{p+1}}{sp} \|f\|^p_p \ve^{-\frac{sp}{2}}.
\] 
On the other hand, one has
\[
\int_0^\ve \frac{1}{t^{\frac{s p}2 +1}} \int_{\bG} P_t\left(|f - f(g)|^p\right)(g) dg dt \leq \left(\underset{\tau\in (0,\ve)}{\sup}\ \tau^{-\frac{p}{2}}\ \int_{\bG} P_\tau\left(|f - f(g)|^p\right)(g) dg \right) \frac{2\ve^{\frac{p}{2}(1-s)}}{p(1 - s)}.
\]
We thus infer that for every $\ve\in (0,\ve_0)$ we have
\begin{equation}\label{ve}
\mathscr N_{s,p}(f)^p \leq \frac{2\ve^{\frac{p}{2}(1-s)}}{p(1 - s)} \left(\underset{\tau\in (0,\ve)}{\sup}\ \tau^{-\frac{p}{2}}\ \int_{\bG} P_\tau\left(|f - f(g)|^p\right)(g) dg \right) + \frac{2^{p+1}}{sp} \|f\|^p_p \ve^{-\frac{sp}{2}}.
\end{equation}
Multiplying by $(1 - s)$ in \eqref{ve} and taking the $\underset{s \to 1^-}{\limsup}$, we find for any $\ve\in (0,\ve_0)$,
\begin{equation}\label{lim12}
\underset{s\to 1^-}{\limsup}\, (1 - s) \mathscr N_{s,p}(f)^p  \le \frac{2}{p} \underset{\tau\in (0,\ve)}{\sup}\ \tau^{-\frac{p}{2}}\ \int_{\bG} P_\tau\left(|f - f(g)|^p\right)(g) dg.
\end{equation}
Passing to the limit as $\ve\to 0^+$ in \eqref{lim12}, we reach the desired conclusion \eqref{ve2}.\\
We are left with the proof of
\begin{equation}\label{ve3}
\underset{s\to 1^-}{\liminf}\ (1 - s)\ \mathscr N_{s,p}(f)^p   \geq  \frac{2}{p}\ \underset{t \to 0^+}{\liminf}\ t^{-\frac{p}{2}}\ \int_{\bG} P_t\left(|f - f(g)|^p\right)(g) dg.
\end{equation}
To do this, for every $0<s<1$ and any $\ve>0$, we write
\begin{align*}
& (1 - s)\ \mathscr N_{s,p}(f)^p \geq (1-s) \int_0^\ve \frac{1}{t^{\frac{s p}2 +1}} \int_{\bG} P_t\left(|f - f(g)|^p\right)(g) dg dt
\\
& \ge (1 -s) \left(\underset{0<\tau<\ve}{\inf}\  \tau^{-\frac{p}{2}} \int_{\bG} P_\tau\left(|f - f(g)|^p\right)(g) dg \right) \int_0^\ve t^{\frac{p}{2}(1-s)-1} dt
\\
& =  \frac{2}{p}\left(\underset{0<\tau<\ve}{\inf}\  \tau^{-\frac{p}{2}} \int_{\bG} P_\tau\left(|f - f(g)|^p\right)(g) dg \right)\ \ve^{\frac{p}{2}(1-s)}. 
\end{align*}
Taking the $\underset{s \to 1^-}{\liminf}$ in the latter inequality, yields
\[
\underset{s\nearrow 1}{\liminf}\, (1-s) \mathscr N_{s,p}(f)^p \geq \underset{0<\tau<\ve}{\inf}\  \frac{2}{p} \tau^{-\frac{p}{2}} \int_{\bG} P_\tau\left(|f - f(g)|^p\right)(g) dg.
\]
If we now take the limit as $\ve\to 0^+$, we reach the desired inequality \eqref{ve3}.
\end{proof}

We are then ready to provide the 

\begin{proof}[Proof of Theorem \ref{T:bbmG}]
The characterisations \eqref{1sp} and, respectively, \eqref{1suno} follow easily from \eqref{charlimsup}, \eqref{charliminf}, and \eqref{chaininfsup} in case $p>1$ and, respectively, from \eqref{supuno}, \eqref{infuno}, and \eqref{chaininfsup} if $p=1$. Moreover, for $f\in W^{1,p}(\bG)$, the limiting behaviour \eqref{2sp} is a trivial consequence of \eqref{2p}, \eqref{2unouno}, and \eqref{chaininfsup}.\\
Finally, if $\bG$ has the property (B) and $f\in BV(\bG)$, the limiting behaviour \eqref{2suno} is a consequence of \eqref{2uno} and \eqref{chaininfsup}.
\end{proof}

We finally close the paper with the

\begin{proof}[Proof of Theorem \ref{T:MS}]
We observe that the heat semigroup $P_t$ satisfies the following three properties:
\begin{itemize}
\item[(a)] $P_t 1=P^*_t1=1$ for all $t>0$ (which is a consequence of (iii) in Proposition \ref{P:prop} and of the symmetry of the heat kernel), and thus in particular
$$
\|P_t f\|_q\leq \|f\|_q \quad \forall\, f\in L^q,\, t>0,\mbox{ and }1\leq q\leq \infty; 
$$
\item[(b)] (ultracontractivity) for every $1<q\leq \infty$ there exists a constant $C_q$ such that
$$
\|P_t f\|_q\leq \frac{C_q}{t^{\frac{Q}{2} \left(1-\frac{1}{q}\right)}} \|f\|_1 \quad \forall\,f\in C_0^\infty\mbox{ and } t>0,
$$
(this is a consequence of Minkowski's integral inequality and the upper Gaussian estimate in \eqref{gauss0});
\item[(c)] the density property in Lemma \ref{L:dens} and the estimate \eqref{besovinbesov} of the embedding $\Bps \subset\mathfrak B_{\sigma,p}(\bG)$.
\end{itemize}
We emphasise that property (a) implies for the spaces $\Bps$ the same asymptotic behaviour as $s\to 0^+$ of the case ${\rm tr} B=0$ of the H\"ormander semigroup treated in \cite[Theorem 1.1]{BGT}.
With properties (a)-(c) we can now follow verbatim the semigroup approach in \cite{BGT} to reach the desired conclusion.

\end{proof}


\bibliographystyle{amsplain}

\end{document}